\documentclass[preprint,3p,12pt]{elsarticle}
\usepackage{amsmath,amssymb,amsthm}
\usepackage{subcaption}
\usepackage{hyperref}
\hypersetup{colorlinks, linkcolor=blue, citecolor=blue, urlcolor=blue}

\newtheorem{thm}{Theorem}

\newtheorem{lem}{Lemma}
\newtheorem{mydef}{Definition}

\def\[#1\]{\begin{align*}#1\end{align*}}

\newcommand\norm[1]{\left\lVert#1\right\rVert}

\newcommand{\bb}{\boldsymbol }

\makeatletter
\def\ps@pprintTitle{%
 \let\@oddhead\@empty
 \let\@evenhead\@empty
 \def\@oddfoot{}%
 \let\@evenfoot\@oddfoot}
\makeatother

\begin{document}

\begin{frontmatter}

\title{Conservative Integrators for Many--body Problems}

\author[UNBC]{Andy T. S. Wan\corref{cor1}}
\ead{andy.wan@unbc.ca}
\author[MUN]{Alexander Bihlo}
\ead{abihlo@mun.ca}
\author[McGill]{Jean-Christophe Nave}
\ead{jcnave@math.mcgill.ca}
\cortext[cor1]{Corresponding author}

\address[UNBC]{Department of Mathematics and Statistics, University of Northern British Columbia,\\ Prince George, BC, V2N 4Z9, Canada}
\address[MUN]{Department of Mathematics and Statistics, Memorial University of Newfoundland,\\ St.\ John's, NL, A1C 5S7, Canada}
\address[McGill]{Department of Mathematics and Statistics, McGill University, Montr\'{e}al, QC, H3A 0B9, Canada}

\begin{abstract}
Conservative symmetric second--order one--step schemes are derived for dynamical systems describing various many--body systems using the Discrete Multiplier Method. This includes conservative schemes for the $n$-species Lotka--Volterra system, the $n$-body problem with radially symmetric potential and the $n$-point vortex models in the plane and on the sphere. In particular, we recover Greenspan--Labudde's conservative schemes for the $n$-body problem. Numerical experiments are shown verifying the conservative property of the schemes and second--order accuracy.   

\end{abstract}

\begin{keyword}
dynamical systems \sep conserved quantity \sep first integral \sep conservative methods \sep Discrete Multiplier Method \sep long-term stability \sep divided difference \sep many--body system \sep Lotka--Volterra equations \sep point vortex equations

\end{keyword}

\end{frontmatter}

\section{Introduction}

The general many--body problem is of great importance in the mathematical sciences. Except for particular configurations such as in ~\cite{calo03a}, the governing equations of $n\geqslant3$ bodies in the classical $n$-body problem cannot be integrated analytically and in general one has to resort to numerical simulations. Similar statements are also true for other nonlinear many--body problems, which also require numerical integrations for large number of bodies. Furthermore, the underlying equations of many--body problems typically have rich geometric structures, such as simplified real-world phenomena including the $n$-species Lotka--Volterra systems and the point vortex system described in this paper. Examples of geometric structures include variational formulations, the existence of first integrals and the invariance under certain coordinate transformations. Note that throughout this paper, we use the terms first integral, invariants of motion, and conserved quantities interchangeably. In order to preserve these structures numerically, special classes of numerical methods, called \textit{geometric numerical integrators}~\cite{blan16a,hair06Ay,leim04Ay} are often employed for this purpose.

Within the field of geometric numerical integration, finding numerical schemes that preserve an underlying Hamiltonian structure of a system of ordinary differential equations (ODEs), so-called symplectic integrators, has been historically of prime interest in recent decades. Indeed, early examples of such numerical schemes date back to the 1950s~\cite{deVo56a}, with other important early contributions found in~\cite{ruth83a} and~\cite{feng86a}, and a first monograph written in the 1990s by~\cite{sanz94a}. For more recent expositions, see for example in~\cite{blan16a,hair06Ay,leim04Ay,mars01a}.

While Hamiltonian systems are important in the mathematical sciences, there are some important restrictions that limit the general applicability of the Hamiltonian framework. Amongst the most important restrictions are systems exhibiting dissipation and systems that cannot be brought easily into a canonical Hamiltonian form, such as dynamical systems with an odd number of degrees of freedom. While for the latter, the more general Poisson geometry and associated Poisson integrators are available~\cite{hair06Ay}, these numerical schemes are not as universally applicable as the symplectic schemes for canonical Hamiltonian systems.

Symplectic integrators also do not preserve arbitrary first integrals of differential equations\footnote{In fact, it is known from \cite{zhon88a} that for Hamiltonian systems without additional first integrals, if a symplectic integrator using an uniform time step size is also energy--preserving, then it is an exact integrator; that is, the discrete flow reproduces the exact flow up to a time reparameterization.}. While the Hamiltonian is nearly conserved over exponentially long time periods~\cite{calv95a} and linear and quadratic first integrals can be preserved by some symplectic Runge--Kutta methods~\cite{hair06Ay,leim04Ay}, higher--order polynomial conserved quantities or first integrals of arbitrary form are generally not preserved with symplectic methods~\cite{feng86a}. If one is interested in the exact preservation of first integrals of  arbitrary forms, then \textit{conservative methods} would need to be utilized, i.e.\ methods that numerically preserved conserved quantities exactly up to machine precision. What sets conservative methods apart from other geometric integrators is that they may possess long-term stability over arbitrarily long time periods~\cite{wan16b}. 

Conservative numerical schemes have also been extensively investigated in the literature. In~\cite{quis08a}, the \textit{average vector field method} was proposed which allows the exact preservation of the energy of arbitrary form in a Hamiltonian system. A generalization of this idea to higher--order schemes using collocation is found in~\cite{hair10a}, with applications to energy preserving schemes for Poisson systems discussed in~\cite{cohe11a}. A further class of general conservative schemes is the \textit{discrete gradient method}, originally proposed in~\cite{quis96a}. This method relies on expressing a first--order system of ODEs in a skew-symmetric gradient form. While some dynamical systems, in particular Hamiltonian systems and so-called Nambu systems~\cite{namb73Ay}, admit natural skew-symmetric gradient representations, many other systems have to be first brought into a skew-gradient form before the discrete gradient method can be applied. Moreover, one main drawback when applying the discrete gradient method to large dimensional systems with multiple invariants is that the order of the associated skew-symmetric tensor increases with the number of conserved quantities.

One further class of exactly conservative methods is given by \textit{projection methods}~\cite{hair06Ay}. Here one applies a standard (usually explicit) integrator over one or more time steps and subsequently projects the resulting numerical approximation onto the manifold spanned by the conserved quantities. As discussed in~\cite{wan16b}, the projection step can become problematic if the manifold spanned by the invariants consists of several connected components, since the projection may then bring the numerical solution onto the wrong connected component.

In~\cite{wan16a}, we have introduced the Discrete Multiplier Method (DMM), which is a general purpose method for finding conservative schemes for dynamical systems with arbitrary forms of conserved quantities. The proposed method rests on discretizing the characteristic~\cite{olve86Ay}, also called conservation law multiplier~\cite{blum10Ay}, of conservation laws so that the discrete conserved quantities hold. This idea was originally proposed in~\cite{wan13Ay} for both PDE and ODE systems, with a systematic framework for constructing conservative finite difference schemes for general ODE systems derived in~\cite{wan16a}. There it was also shown that the average vector field method corresponds to a special choice of the discretization of the conservation law multiplier of Hamiltonian systems. Several examples of conservative schemes for classical dynamical systems were presented in~\cite{wan16a}, but many--body systems were not considered there. As several important dynamical systems are in fact many--body problems, the purpose of this paper is to demonstrate that the constructive framework proposed in~\cite{wan16a} is also suitable for large dynamical systems.

We note here that for the purpose of the present paper, `many--body systems' refers to dynamical systems with at most a few thousand degrees of freedom. While the DMM for finding conservative integrators is not dependent on the number of degrees of freedom of the underlying dynamical system, the resulting conservative schemes are typically implicit. As such, a practical implementation of these schemes generally relies on an (efficient) implicit solver which renders the case of very many bodies (i.e.\ millions and more) computationally challenging. We do not aim to address this computational challenge in the present work, where we exclusively use a standard fixed point iteration for solving these implicit conservative schemes, and reserve a more in-depth study of this computational problem for future work.

The further organization of this paper is as follows. In Section~\ref{sec:TheoryConservativeIntegrators}, we give a brief review of DMM for conservative discretizations as originally proposed in~\cite{wan13Ay}. We then propose several second--order conservative schemes derived using DMM for many--body systems in the following sections. Specifically, Section~\ref{sec:ConservativeIntegratorsLV} is devoted to a conservative schemes for the general $n$-species Lotka--Volterra system of population dynamics. In Section~\ref{sec:ConservativeIntegratorsGeneralNBody}, we present a conservative scheme for the $n$-body problem with general radially symmetric potential and recover Greenspan--Labudde's conservative scheme~\cite{LG75,gree04a}. The celestial $n$-body problem and the Lennard--Jones potential from molecular dynamics are considered as special cases. Section~\ref{sec:ConservativeIntegratorsPointVortex} details conservative schemes for the $n$-point vortex model on the plane and on the sphere. Section~\ref{sec:NumericsManyBody} features numerical results of the various schemes derived in this paper. Finally, we make some concluding remarks in Section~\ref{sec:ConclusionsManyBody}. \ref{sec:Appendix} includes theoretical verifications of conservative properties of all the schemes presented and shows that they are all second--order accurate.

\section{Construction of exactly conservative integrators}\label{sec:TheoryConservativeIntegrators} 

Before discussing the theory of DMM presented in \cite{wan16a}, we first fix some notations which will be used throughout this article. 

\subsection{Notations and conventions}
Let $U\subset \mathbb{R}^n$ and $V\subset \mathbb{R}^m$ be open subsets where here and in the following $n,m,p \in \mathbb{N}$. $f\in C^p(U\rightarrow V)$ means $f$ is a $p$-times continuously differentiable function with domain in $U$ and range in $V$. We often use boldface to indicate a vector quantity $\bb f$.  If $\bb f\in C^1(U\rightarrow V)$, $\partial_{\bb x} \bb f:=\left[\frac{\partial f_i}{\partial x_j}\right]$ denotes the Jacobian matrix. Let $I\subset \mathbb{R}$ be an open interval and let $\bb x\in C^1(I\rightarrow U)$, $\dot{\bb x}$ denotes the derivative with respect to time $t\in I$. Also if $\bb x\in C^p(I\rightarrow U)$, $\bb x^{(q)}$ denotes the $q$-th time derivative of $\bb x$ for $1\leq q\leq p$.  For brevity, the explicit dependence of $\bb x$ on $t$ is often omitted with the understanding that $\bb x$ is to be evaluated at $t$. If $\bb \psi\in C^1(I\times U\rightarrow V)$, $D_t \bb\psi$ denotes the total derivative with respect to $t$, and $\partial_t \bb\psi$ denotes the partial derivative with respect to $t$. $M_{m\times n}(\mathbb{R})$ denotes the set of all $m\times n$ matrices with real entries. 

\subsection{Conserved quantities of quasilinear first--order ODEs}
Consider a quasilinear first--order system of ODEs,
\begin{align}
\dot{\bb x}(t) &= \bb f(t,\bb x), \label{odeEqn} \\
\bb x(t_0)&=\bb x_0. \nonumber
\end{align} where $t\in I$, $\bb x=(x_1(t),\dots, x_n(t)) \in U$. 
For $1\leq p\in \mathbb{N}$, if $\bb f \in C^{p-1}(I\times U\rightarrow\mathbb{R}^n)$ and is Lipschitz continuous in $U$, then standard ODE theory implies there exists an unique solution $\bb x\in C^p(I\rightarrow U)$ to the first--order system \eqref{odeEqn} in a neighborhood of $(t_0,\bb x_0)\in I\times U$.

\begin{mydef} Let $m\in\mathbb{N}$ with $1\leq m\leq n$. A vector-valued function $\bb \psi \in C^1(I\times U\rightarrow \mathbb{R}^m)$ is a vector of conserved quantities\footnote{By quasilinearity of \eqref{odeEqn}, it suffices to consider conserved quantities depending only on $t,\bb x$, see \cite{wan16a}.} (or equivalently first integrals) if 
\begin{align}
D_t \bb \psi(t,\bb x) = \bb 0, \text{ for any $t\in I$ and $C^1(I\rightarrow U)$ solution } \bb x \text{ of }\eqref{odeEqn}. \label{consQ}
\end{align} In other words, $\bb \psi(t,\bb x)$ is constant on any $C^1(I\rightarrow U)$ solution $\bb x$ of \eqref{odeEqn}.
\end{mydef}

A generalization of integrating factors is known as \emph{characteristics} by \cite{olve86Ay} or equivalently, \emph{conservation law multipliers} by \cite{blum10Ay}. We will adopt the terminology of conversation law multiplier or just multiplier when the context is clear. 
\begin{mydef}
Let $m\in\mathbb{N}$ with $1\leq m\leq n$ and $U^{(1)}$ be an open subset of $\mathbb{R}^n$. A conservation law multiplier of $\bb F$ is a matrix-valued function $\Lambda \in C(I\times U \times U^{(1)}\rightarrow M_{m\times n}(\mathbb{R}))$ such that there exists a function $\bb \psi \in C^1(I\times U\rightarrow \mathbb{R})$ satisfying,
\begin{align}
\Lambda(t,\bb x, \dot{\bb x})(\dot{\bb x}(t)-\bb f(t,\bb x)) = D_t \bb \psi(t,\bb x), \text{ for $t\in I$, }\bb x \in C^1(I\rightarrow U). \label{multEqn}
\end{align}
\end{mydef}
Here, we emphasize that condition \eqref{multEqn} is satisfied as an identity for arbitrary $C^1$ functions $\bb x$; in particular $\bb x$ need not be a solution of \eqref{odeEqn}. It follows from the definition of conservation law multiplier that existence of multipliers implies existence of conservation laws. Conversely, given a known vector of conserved quantities $\bb \psi$, there can be many conservation law multipliers which correspond to $\bb \psi$. It was shown in \cite{wan16a} that it suffices to consider multipliers of the form $\Lambda(t,\bb x)$ where a one-to-one correspondence exists between conservation law multipliers and conserved quantities of \eqref{odeEqn}. 
\begin{thm}[Theorem 4 of \cite{wan16a}]
Let $\bb \psi\in C^1(I\times U\rightarrow \mathbb{R}^m)$. Then there exists a unique conservation law multiplier of \eqref{odeEqn} of the form $\Lambda \in C(I\times U\rightarrow M_{m\times n}(\mathbb{R}))$ associated with the function $\bb \psi$ if and only if $\bb \psi$ is a conserved quantity of \eqref{odeEqn}. And if so, $\Lambda$ is unique and satisfies for any $t\in I$ and $ \bb x\in C^1(I\rightarrow U)$,
\begin{subequations}
\begin{align}
\Lambda(t,\bb x) = \partial_{\bb x} \bb \psi(t,\bb x), \label{multCond1} \\
\Lambda(t,\bb x) \bb f(t,\bb x) = -\partial_t \bb \psi(t,\bb x). \label{multCond2}
\end{align}
\end{subequations}
\end{thm}
\noindent To construct conservative methods for \eqref{odeEqn} with conserved quantities \eqref{consQ}, we shall discretize the time interval $I$ by a uniform time size $\tau\in\mathbb{R}$, i.e. $t^{k+1}= t^k+\tau$ for $k\in \mathbb{N}$, and focus on one--step conservative methods\footnote{Analogous results hold for variable time step sizes and multi-step methods, see \cite{wan16a} for more details.}. First, we recall some definitions from \cite{wan16a}.

\begin{mydef}
Let $W$ be a normed vector space, such as $\mathbb{R}^m$ with the Euclidean norm or $M_{m\times n}(\mathbb{R})$ with the operator norm. A function $g^\tau:I\times U\times U \rightarrow W$ is called a one--step function if $g^\tau$ depends only on $t^k\in I$ and the discrete approximations $\bb x^{k+1}, \bb x^{k} \in U$.
\end{mydef}
\begin{mydef}
A sufficiently smooth one--step function $g^\tau:I\times U\times U\rightarrow W$ is consistent to a sufficiently smooth $g: I\times U\times U^{(1)}\rightarrow W$ if for any $\bb x\in C^2(I\rightarrow U)$, there is a constant $C>0$ independent of $\tau$ so that $\norm{g(t^k, \bb x(t^k), \dot{\bb x}(t^k))-g^\tau(t^k, \bb x(t^{k+1}), \bb x(t^k))}_W\leq C\norm{\bb x}_{C^2([t^{k}, t^{k+1}])}\tau,
$ where $\norm{\bb x}_{C^2([t^{k}, t^{k+1}])} := \displaystyle \max_{0\leq i\leq 2} \norm{\bb x^{(i)}}_{L^\infty([t^{k}, t^{k+1}])}$. If so, we write $g^\tau = g + \mathcal{O}(\tau)$.
\end{mydef}

We shall be considering the following consistent one--step functions for $\dot{\bb x}, D_t \bb \psi, \partial_t \bb \psi$:
\begin{align}
D_t^\tau \bb x(t^k, \bb x^{k+1}, \bb x^k)&:= \frac{\bb x^{k+1}-\bb x^k}{\tau}= \dot{x}+\mathcal{O}(\tau), \label{discFunc1} \\
D_t^\tau \bb \psi(t^k, \bb x^{k+1}, \bb x^k)&:= \frac{\bb \psi(t^{k+1}, \bb x^{k+1})-\bb \psi(t^{k}, \bb x^{k})}{\tau}= D_t\bb \psi+\mathcal{O}(\tau), \label{discFunc2} \\
\partial_t^\tau \bb \psi(t^k, \bb x^{k+1}, \bb x^k)&:= \frac{\bb \psi(t^{k+1}, \bb x^{k})-\bb \psi(t^{k}, \bb x^{k})}{\tau}= \partial_t \bb \psi+\mathcal{O}(\tau). \label{discFunc3} 
\end{align}
\begin{mydef}
Let $\bb f^\tau$ be a consistent one--step function to $\bb f$. We say that the one--step method,
\begin{equation}
D_t^\tau \bb x(t^k, \bb x^{k+1}, \bb x^k) = \bb f^\tau(t^k, \bb x^{k+1}, \bb x^k) \label{discEqn}
\end{equation} 
is conservative in $\bb \psi$, if $\bb \psi(t^{k+1}, \bb x^{k+1})=\bb \psi(t^k, \bb x^k)$ on any solution $\bb x^{k+1}$ of \eqref{discEqn} and $k\in \mathbb{N}$.
\end{mydef}
We now state two key conditions from \cite{wan16a} for constructing conservative one--step methods, which can be seen as a discrete analogue of \eqref{multCond1} and \eqref{multCond2}.

\begin{thm}[Theorem 17 of \cite{wan16a}] Let $D_t^\tau \bb x, D_t^\tau \bb \psi, \partial_t^\tau \bb \psi$ be as defined in \eqref{discFunc1}--\eqref{discFunc3}, and let $\Lambda$ be the conservation law multiplier of \eqref{odeEqn} associated with a conserved quantity $\bb \psi$. If $\bb f^\tau$ and $\Lambda^\tau$ are consistent one--step functions to $\bb f, \Lambda$ satisfying
\begin{subequations}
\begin{align}
\Lambda^\tau D^\tau_t \bb x &= D^\tau_t \bb \psi - \partial^\tau_t \bb \psi, \label{discMultCond1}\\
\Lambda^\tau \bb f^\tau &= -\partial^\tau_t \bb \psi, \label{discMultCond2}
\end{align}
\end{subequations}
then the one--step method defined by \eqref{discEqn} is conservative in $\bb \psi$. \label{thm:main}
\end{thm}
In \cite{wan16a}, condition \eqref{discMultCond1} was solved by the use of divided difference calculus and \eqref{discMultCond2} was solved using a local matrix inversion formula. In the following many--body problems, we shall directly verify \eqref{discMultCond1} and \eqref{discMultCond2} for specific choices of $\bb f^\tau$ and $\Lambda^\tau$.

Lastly, we recall a well-known result for even order of accuracy of symmetric schemes. For more details, see Chapter II.3 of \cite{hair06Ay}.

\begin{mydef}[Symmetric Schemes~\cite{hair06Ay}]
Let $\Phi^\tau$ be the discrete flow of a one--step numerical method for system~\eqref{odeEqn} with time step $\tau$. The associated adjoint method $(\Phi^\tau)^*$ of the one--step method $\Phi^\tau$ is the inverse of the original method with reversed time step $-\tau$, i.e.\ $(\Phi^\tau)^* = (\Phi^{-\tau})^{-1}$. A method is symmetric if $(\Phi^\tau)^*=\Phi^\tau$.
\label{def:symmetricMethod}
\end{mydef}

\begin{thm}[Theorem II-3.2 of \cite{hair06Ay}]
A symmetric method is of even order.
\label{thm:evenOrder}
\end{thm}
In order words, combining with Theorem \ref{thm:main}, the one--step conservative schemes are at least second--order accurate if they are symmetric.

\section{Examples of DMM for many--body systems}

In the section, we review some examples of many--body systems from population dynamics, classical mechanics, molecular dynamics and fluid dynamics. Specifically, we present conservative schemes derived using DMM for the $n$-species Lotka--Volterra systems, the $n$-body problem involving the gravitational potential and the Lennard--Jones potential, and the $n$-point vortex problem on the plane and on the unit sphere. For brevity and clarity, we have included all the details of calculations for derivations and verifications in \ref{sec:Appendix}.

\subsection{Conservative schemes for Lotka--Volterra systems}\label{sec:ConservativeIntegratorsLV}

The $n$-species Lotka--Volterra equations describe a simplified dynamics among $n$ competing species interacting in an environment~\cite{hofb98a}. Specifically, we consider the $n$-species Lotka--Volterra system given in the form of,
\begin{align}
\boldsymbol F(\bb x, \dot{\bb x}):= \begin{bmatrix}\dot{x}_i - x_i \sum_{j=1}^n a_{ij}(x_j-\xi_j)\end{bmatrix}_{1\leq i\leq n} = {\bb 0}, \label{LVsys}
\end{align} where $\bb x = (x_1,\dots, x_n)^T$ is the population of each species with $x_i>0$, $A=[a_{ij}]$ is an $n\times n$ interaction matrix with real entries and $\bb \xi = (\xi_1,\dots, \xi_n)^T$ is a fixed point of the system. 
It is known from \cite{schi03a} that \eqref{LVsys} has the conserved quantity
\vskip -2mm
\begin{equation}\label{eq:ConservedQuantityLotkaVolterraSystem}
V(\bb x) := \sum_{i=1}^n d_i(\xi_i \log x_i-x_i),
\end{equation}
if there exists an $n\times n$ real diagonal matrix $D=\text{diag}(d_1,\dots, d_n)$ such that $DA$ is skew-symmetric. 

A conservative scheme for \eqref{LVsys} which preserves $V$ numerically was derived using DMM and is given by,
\begin{align}
\boxed{
\boldsymbol F^\tau(\bb x^{k+1}, \bb x^{k}) := \begin{bmatrix}\dfrac{x_i^{k+1}-x_i^k}{\tau} - x_i^\tau\sum\limits_{j=1}^n a_{ij}x_j^\tau\left(1- \dfrac{\xi_j}{x_j^k}g\left(\dfrac{x_j^{k+1}}{x_j^k}\right)\right) \end{bmatrix}_{1\leq i\leq n}=\bb 0,
}
\label{LVDisc}
\end{align}
where $g(z) := \dfrac{\log z}{z-1}$ and $x_i^\tau$ is any consistent discretization of $x_i$ (i.e. $x_i^\tau \rightarrow x_i$, as $\tau \rightarrow 0$). It is interesting to note that \eqref{LVDisc} is consistent to \eqref{LVsys} since $g(z) = 1 - \frac{1}{2}(z-1) + \mathcal{O}((z-1)^2) \rightarrow 1$ as $\dfrac{x_j^{k+1}}{x_j^k} \rightarrow 1$ with $\tau\rightarrow 0$. 

The simplest consistent choices of $x_i^\tau$ would be $x_i^\tau := x_i^k$ or $x_i^{k+1}$, which will lead to first--order conservative schemes. Thus, for improved accuracy with similar computational costs, we choose $x_i^\tau$ so that the resulting scheme \eqref{LVDisc} is symmetric, which will lead to a second--order scheme according to Theorem \ref{thm:evenOrder}. In particular, we show in \ref{app:LV} that \eqref{LVDisc} is conservative and is symmetric if $x_i^\tau$ itself is symmetric. Specifically, choosing $x_i^\tau = \overline{x_i} := \dfrac{1}{2}\left(x_i^k+x_i^{k+1}\right)$ leads to an ``Arithmetic mean DMM" scheme for \eqref{LVsys}. Since the phase variables $x_i$ for \eqref{LVsys} are nonnegative, another choice is $x_i^\tau := \sqrt{x_i^{k}x_i^{k+1}}$ or a ``Geometric mean DMM" scheme for \eqref{LVsys}. 

\subsection{Many--body problem with pairwise radial potentials}\label{sec:ConservativeIntegratorsGeneralNBody}

The many--body problem with pairwise radial potential, i.e.\ with conservative forces that only depend on the radial difference of each two point masses, is one of the most fundamental models in classical mechanics. It describes, in an idealized fashion, numerous physical phenomena, including the motion of planets in a solar system and atoms of molecules. 

We consider the many--body problem as the Hamiltonian system with $n$ particles in $\mathbb{R}^3$ with radial interaction potentials\footnote{There are in general $6n$ unknowns for \eqref{nBodySys}. In particular, the two body problem is integrable since there are 10 constants of motion with an additional two conserved quantities provided by the Laplace--Runge--Lenz vector.},
\begin{equation}
\bb F(\bb q, \bb p, \dot{\bb q}, \dot{\bb p}):=
\begin{pmatrix}
 \begin{bmatrix}\dot{\bb q}_i-\dfrac{\bb p_i}{m_i}\end{bmatrix}_{1\leq i\leq n}  \\
 \begin{bmatrix}\dot{\bb p}_i+\sum\limits_{j=1, j\neq i}^n \dfrac{\partial V_{ij}}{\partial q_{ij}}(q_{ij})\dfrac{\bb q_i-\bb q_j}{q_{ij}}\end{bmatrix}_{1\leq i\leq n}
 \end{pmatrix} = \bb 0, \label{nBodySys}
\end{equation} where $\bb q = (\bb q_1,\dots, \bb q_n)^T, \bb p = (\bb p_1,\dots, \bb p_n)^T$ with $\bb q_i \in \mathbb{R}^3, \bb p_i \in \mathbb{R}^3, m_i \in \mathbb{R}_+$ as the position, momenta and mass of the $i$-th particle. For each distinct pair $(i,j)$ of particles, $q_{ij}:= |\bb q_i-\bb q_j|= q_{ji}$ denotes their Euclidean distance and $V_{ij}$ is their radial pairwise potential energy such that $V_{ij}=V_{ji}$. From classical mechanics, it is well-known that there are ten constants of motion for \eqref{nBodySys}. Specifically, they are the Hamiltonian $H$, total linear momentum $\bb P$, total angular momentum $\bb L$ and initial center of mass $\bb C$ -- given by,
\begin{align}\label{eq:ConservedQuantitiesManyBodyProblem}
\begin{split}
H(\bb q, \bb p) &:= \sum_{i=1}^n \dfrac{\bb p_i^T \bb p_i}{2m_i}+\sum_{1\leq i<j\leq n} V_{ij}(|\bb q_i-\bb q_j|)\\
\bb P(\bb q, \bb p) &:= \sum_{i=1}^n \bb p_i \\
\end{split}
\begin{split}
\bb L(\bb q, \bb p) &:= \sum_{i=1}^n \bb q_i \times \bb p_i \\
\bb C(t,\bb q, \bb p) &:= \dfrac{1}{M}\left(\sum_{i=1}^n m_i \bb q_i\right)-\dfrac{\bb P}{M} t,%
\end{split}
\end{align}
where $M = \sum_{i=1}^n m_i$ is the total mass of the system. 

A conservative scheme for \eqref{nBodySys} which preserves all ten first integrals was derived using DMM and is given by,
\begin{align}
\boxed{
\bb F^\tau(\bb q^{k+1},\bb p^{k+1}, \bb q^k, \bb p^k) := \begin{pmatrix}
\begin{bmatrix}
\dfrac{\bb q^{k+1}-\bb q^k}{\tau}- \dfrac{\overline{\bb p_i}}{m_i}\end{bmatrix}_{1\leq i\leq n}\\
\begin{bmatrix}
\dfrac{\bb p^{k+1}-\bb p^k}{\tau}+\sum\limits_{j=1, j\neq i}^n \dfrac{\Delta V_{ij}}{\Delta q_{ij}}\dfrac{1}{\overline{q_{ij}}}(\overline{\bb q_i}-\overline{\bb q_j})\end{bmatrix}_{1\leq i\leq n}
\end{pmatrix} = \bb 0.
}\label{nBodyDisc}
\end{align}
We note that the discretization~\eqref{nBodyDisc} was previously reported in~\cite[Section 2.2]{gree04a}, although there no constructive derivation was given. It is thus the added benefit of DMM that conservative schemes can be constructed systematically as detailed in the Appendix.

Here, scalar or vector quantities with a line above denote the arithmetic mean of the quantity at $t_k$ and $t_{k+1}$. Similar to the previous Lotka--Volterra example, the reason for this symmetric choice is due to Theorem \ref{thm:evenOrder}. As shown in the Appendix, \eqref{nBodyDisc} is conservative and symmetric provided $\frac{\Delta V_{ij}}{\Delta q_{ij}}$ is symmetric. 

We now present some particular cases of the potential~$V$ that specify the discretizations~\eqref{nBodyDisc} for physically relevant problems in celestial mechanics and molecular dynamics.

\subsubsection{Gravitational Potential}

The classical Newtonian form of the many--body problem as applies to the solar system is given by the following specification of the pairwise radial potential,
\begin{align*}
&V_{ij}(q_{ij}) = -\dfrac{Gm_i m_j}{q_{ij}},
\end{align*}
where $G$ is Newton's gravitational constant. For the numerical scheme~\eqref{nBodyDisc}, the respective divided difference for the potential thus simplifies to
\[
&\boxed{\frac{\Delta V_{ij}}{\Delta q_{ij}} = \frac{Gm_i m_j}{q_{ij}^{k+1} q_{ij}^k}},
\]
which is symmetric under the permutation $k\leftrightarrow k+1$.
\subsubsection{Lennard--Jones Potential}

In classical molecular dynamics, forces are modeled to be attractive when particles are far from each other and repulsive when they are close. A classical example for a potential in molecular dynamics is the Lennard--Jones potential, given by
\begin{align*}
&V_{ij}(q_{ij}) = 4\epsilon\left(\dfrac{\sigma^{12}}{q_{ij}^{12}}-\dfrac{\sigma^{6}}{q_{ij}^{6}}\right),
\end{align*}
where $\epsilon$ and $\sigma$ are the potential well depth and the distance where the potential becomes zero, respectively. For this particular form of the potential, the divided difference for the potential in~\eqref{nBodyDisc} becomes
\begin{align}
&\boxed{\frac{\Delta V_{ij}}{\Delta q_{ij}} = 4\epsilon\left(-\frac{\sigma^{12}}{q_{ij}^{k+1}q_{ij}^k}\sum_{l=0}^{11} (q_{ij}^{k+1})^{l-11} (q_{ij}^k)^{-l}+\frac{\sigma^{6}}{q_{ij}^{k+1}q_{ij}^k}\sum_{l=0}^{5} (q_{ij}^{k+1})^{l-5} (q_{ij}^k)^{-l}\right)},
\label{eqn:LJ_dividedDiff}
\end{align} which is again symmetric under the permutation $k\leftrightarrow k+1$.

\subsection{Point vortex problem}\label{sec:ConservativeIntegratorsPointVortex}

The simplified modeling of the continuous description of fluid mechanics on the plane by an ensemble of point vortices dates back more than 150 years, and was first introduced by Helmholtz~\cite{helm58a}, with other important early contributions due to Kirchhoff who in particular first established the Hamiltonian representation of the equations of point vortex dynamics~\cite[Lecture 20]{kirc83a}, see also~\cite{Newton01,aref07a} for more detailed reviews.

We first consider the classical $n$-point vortex problem in the plane before discussing the $n$-point vortex problem on the unit sphere. While in the planar case the point vortex equations are a canonical Hamiltonian system, in the spherical case the equations constitute a non-canonical Hamiltonian system. From the point of view of geometric numerical integration, standard symplectic integrators can be used in the planar case, see results in~\cite{chan90a,pull91a,scov91a}, while the spherical case requires the use of Poisson integrators, with some recent examples found in~\cite{myer16a,vank14a}. Here, we show the DMM framework provides exactly conservative integrators for both cases.

\subsubsection{Planar case}

Consider the $n$-point vortex problem on the plane,
\begin{align}
\boldsymbol F(\bb x, \bb y, \dot{\bb x}, \dot{\bb y}) := 
\begin{pmatrix}
\begin{bmatrix}
\dot{x}_i+\dfrac{1}{2\pi}\sum\limits_{j=1, j\neq i}^n \Gamma_j \dfrac{y_{ij}}{ r_{ij}^2}
\end{bmatrix}_{1\leq i \leq n}\\
\begin{bmatrix}
\dot{y}_i-\dfrac{1}{2\pi}\sum\limits_{j=1, j\neq i}^n \Gamma_j \dfrac{x_{ij}}{ r_{ij}^2}
\end{bmatrix}_{1\leq i \leq n}
\end{pmatrix} =\boldsymbol 0, \label{pvPlane}
\end{align} where $\bb x=(x_1,\dots, x_n)^T$ and $\bb y=(y_1,\dots, y_n)^T$ with $(x_i,y_i)$ being the position of the $i$-th point vortex on the plane, and $\Gamma_i$ is the vorticity strength of the $i$-th vortex. We abbreviate $x_{ij}:=x_i-x_j$ and $y_{ij}:=y_i-y_j$ and $r_{ij}=\sqrt{x_{ij}^2+y_{ij}^2}$. It is well-known that \eqref{pvPlane} possesses four conserved quantities -- linear momentum $\bb P$, angular momentum $L$ and the Hamiltonian $H$ -- given by,
\begin{align}\label{eq:ConservedQuantitiesPlanarPointVortex}
\begin{split}
\bb P(\bb x, \bb y) &:= \begin{pmatrix} 
					\sum\limits_{i=1}^n \Gamma_i x_i \\
					\sum\limits_{i=1}^n \Gamma_i y_i 
				  \end{pmatrix},
\end{split}
\begin{split}
L(\bb x, \bb y) &:= \sum\limits_{i=1}^n \Gamma_i (x_i^2+y_i^2), \\
H(\bb x, \bb y) &:= -\dfrac{1}{2\pi}\sum\limits_{1\leq i< j\leq n} \Gamma_i \Gamma_j \log  r_{ij}.
\end{split}
\end{align}

A conservative scheme that preserves these four conserved quantities was found using DMM and is given by,
\begin{align}\boxed{
\boldsymbol F^\tau(\bb x^{k+1}, \bb y^{k+1}, \bb x^{k}, \bb y^{k}) := \begin{pmatrix} \begin{bmatrix}
\dfrac{x_i^{k+1}-x_i^{k}}{\tau} + \dfrac{1}{2\pi}\sum\limits_{j=1, j\neq i}^n 
\dfrac{\Gamma_j \overline{y_{ij}}}{( r_{ij}^k)^2}g\left(\left(\frac{r_{ij}^{k+1}}{r_{ij}^{k}}\right)^2\right) \end{bmatrix}_{1\leq i\leq n} \\ \begin{bmatrix}
\dfrac{y_i^{k+1}-y_i^{k}}{\tau}- \dfrac{1}{2\pi}\sum\limits_{j=1, j\neq i}^n  \dfrac{\Gamma_j \overline{x_{ij}}}{( r_{ij}^k)^2}g\left(\left(\frac{r_{ij}^{k+1}}{r_{ij}^{k}}\right)^2\right) \end{bmatrix}_{1\leq i\leq n}
\end{pmatrix} = \boldsymbol 0, \label{pvPlaneDisc}}
\end{align} where $g(z)=\dfrac{\log z}{z-1}$ as in the Lotka--Volterra example. We show in the Appendix that \eqref{pvPlaneDisc} is conservative and symmetric. 

\subsubsection{Spherical case}

The $n$-point vortex problem on the unit sphere is governed by the equations
\begin{align}
\boldsymbol F(\bb x, \dot{\bb x}) := 
\dot{\bb x}_i-\dfrac{1}{4\pi}\sum\limits_{j=1, j\neq i}^n \Gamma_j \dfrac{{\bb x}_j\times {\bb x}_i}{1-{\bb x}_i\cdot {\bb x_j}}
 =\boldsymbol 0, \label{pvSphere}
\end{align} where $\bb x=({\bb x}_1,\dots, {\bb x}_n)^T$ with ${\bb x}_i$ being the position of the $i$-th point vortex on the sphere and $\Gamma_i$ being the vortex strength of the $i$-th vortex. The point vortex equations on the unit sphere~\eqref{pvSphere} possess four conserved quantities, given by the Noether momentum $\bb P$ and the Hamiltonian $H$, which are
\begin{align}\label{eq:ConservedQuantitiesSphericalPointVortex}
\begin{split}
\bb P(\bb x) &:= \sum\limits_{i=1}^n \Gamma_i {\bb x}_i,
\end{split}
\begin{split}
H(\bb x) &:= -\dfrac{1}{4\pi}\sum\limits_{1\leq i< j\leq n} \Gamma_i \Gamma_j \log  (2-2{\bb x}_i\cdot {\bb x}_j).
\end{split}
\end{align}

The conservative discretization for~\eqref{pvSphere} was found using DMM and is given by
\begin{align}\boxed{
\boldsymbol F^\tau(\bb x^{k+1}, \bb x^{k}) := \dfrac{{\bb x}_i^{k+1}-{\bb x}_i^k}{\tau}-\dfrac{1}{4\pi}\sum\limits_{1\leq j \leq n, j\neq i} 
\Gamma_j \dfrac{\overline{{\bb x}_j}\times\overline{{\bb x}_i}}{1-{{\bb x}_i^k}\cdot{{\bb x}_j^k}}g\left(\frac{1-{\bb x}_i^{k+1}\cdot{\bb x}_j^{k+1}}{1-{\bb x}_i^{k}\cdot{\bb x}_j^{k}}\right) = \boldsymbol 0, \label{pvSphereDisc}}
\end{align}
where $g(z)=\dfrac{\log z}{z-1}$, as in the planar case. In the Appendix, \eqref{pvSphereDisc} is also shown to be conservative and symmetric.

\section{Numerical results}\label{sec:NumericsManyBody}

In this section, we present numerical results for the conservative schemes derived using DMM. We compare numerically the conservative property and second--order accuracy with some classical and symplectic methods. Specifically, as the derived DMM schemes are second--order implicit schemes, we compare all examples with the implicit Midpoint method, which has similar computational cost and is a second--order symplectic method with favourable long-term properties \cite{hair06Ay}. Moreover, we also compare with some explicit methods when applicable, such as the standard fourth-order explicit Runge--Kutta method and the St\"{o}rmer--Verlet method, which is a second--order symplectic method.

In the following, all implicit methods were solved by a fixed point iteration using the right hand side of the respective schemes with an absolute tolerance of $10^{-14}$, unless otherwise stated. For a final time $T$, we have used a uniform time step of size $\tau$ with a total number of $N$ time steps. For a conserved quantity $\psi$, we denote the $\ell^\infty$ norm in discrete time of the error by
\[
\text{Error}[\psi(t,\bb x)]:= \max_{k=1,\dots, N} |\psi(t^k,\bb x^k)-\psi(0,\bb x^0)|.
\]

\subsection{Lotka--Volterra systems}

We begin with the Lotka--Volterra systems of Section \ref{sec:ConservativeIntegratorsLV}. For reference, we compare the arithmetic and geometric mean DMM schemes of \eqref{LVDisc} with the Midpoint method and the standard explicit fourth-order Runge--Kutta method. 

Specifically, we consider the three--species non-degenerate Lotka--Volterra system represented by the $3\times3$ interaction matrix
\[
A=\left(\begin{array}{ccc}
1 & 1 & 1\\
0 & 0 & -2\\
0 & 1 & 0
\end{array}\right),
\] and a fixed point of ${\bf \xi}=\left(\frac{1}{2},\frac{1}{2},\frac{1}{2}\right)^{\rm T}$. This system is non-degenerate since $\det\left(A\right)\neq0.$ Additionally,
since $DA+A^{\rm T}D=0$ for $D=\text{diag}\left(0,1,2\right)$, it possesses
\emph{one} conserved quantity $V\left(\bb x\right) = \frac{1}{2}\log y-y+2\left(\frac{1}{2}\log z-z\right)$, as defined in Section \ref{sec:ConservativeIntegratorsLV}. This is precisely the conserved quantity that DMM was constructed to preserve exactly.

We choose the initial conditions of ${\bb x_{0}}=\left(\frac{1}{10},\frac{1}{10},\frac{1}{10}\right)^{\rm T}$ with a final time of $T=50$ and $\tau=5\cdot10^{-2}$. For the implicit methods considered, we have applied one step of the forward Euler method as the initial guess for the fixed point iterations with an absolute tolerance of $10^{-15}$. 

The convergence plot in Figure~\ref{LV 3 Species--non degenerate--Convergence} confirms that all methods has indeed the expected accuracy orders. In particular, the Arithmetic mean and Geometric mean DMM are confirmed to be second order. Figure \ref{LV 3 Species--non degenerate--CQError} shows that the error in time of the conserved quantity $V\left(\bb x\right)$ for all four methods considered. We see the exact conservation (up to machine precision) for the two DMM schemes, in contrast to the Midpoint method and the fourth-order Runge--Kutta method. As expected, we only observe a slow growth of error in $V(\bb x)$ for the DMM schemes due to accumulation of round-off errors and the nonzero tolerance imposed by the fixed point iterations. Table \ref{LV-3species-non_degenerate} summarizes the $\ell^\infty$ error in $V(\bb x)$ for all four methods.

\begin{table}[!ht]
\centering
\begin{tabular}{|c|c|c|c|c|}
\hline 
 Method & Midpoint & RK4 & Arith. Mean DMM& Geo. Mean DMM \tabularnewline
\hline 
\hline 
$\text{Error}[V(\bb x)]$ & $1.76\cdot10^{-3}$ & $3.68\cdot10^{-6}$ & $1.78\cdot10^{-15}$ & $6.22\cdot10^{-15}$\tabularnewline
\hline 
\end{tabular}
\caption{Error in conserved quantity $V\left(\bb x\right)$ for the non-degenerate three--species Lotka--Volterra model.}
\label{LV-3species-non_degenerate}
\end{table}

\begin{figure}[!ht]
\centering
\begin{subfigure}{0.49\textwidth} 
  \centering
  \includegraphics[width=\linewidth]{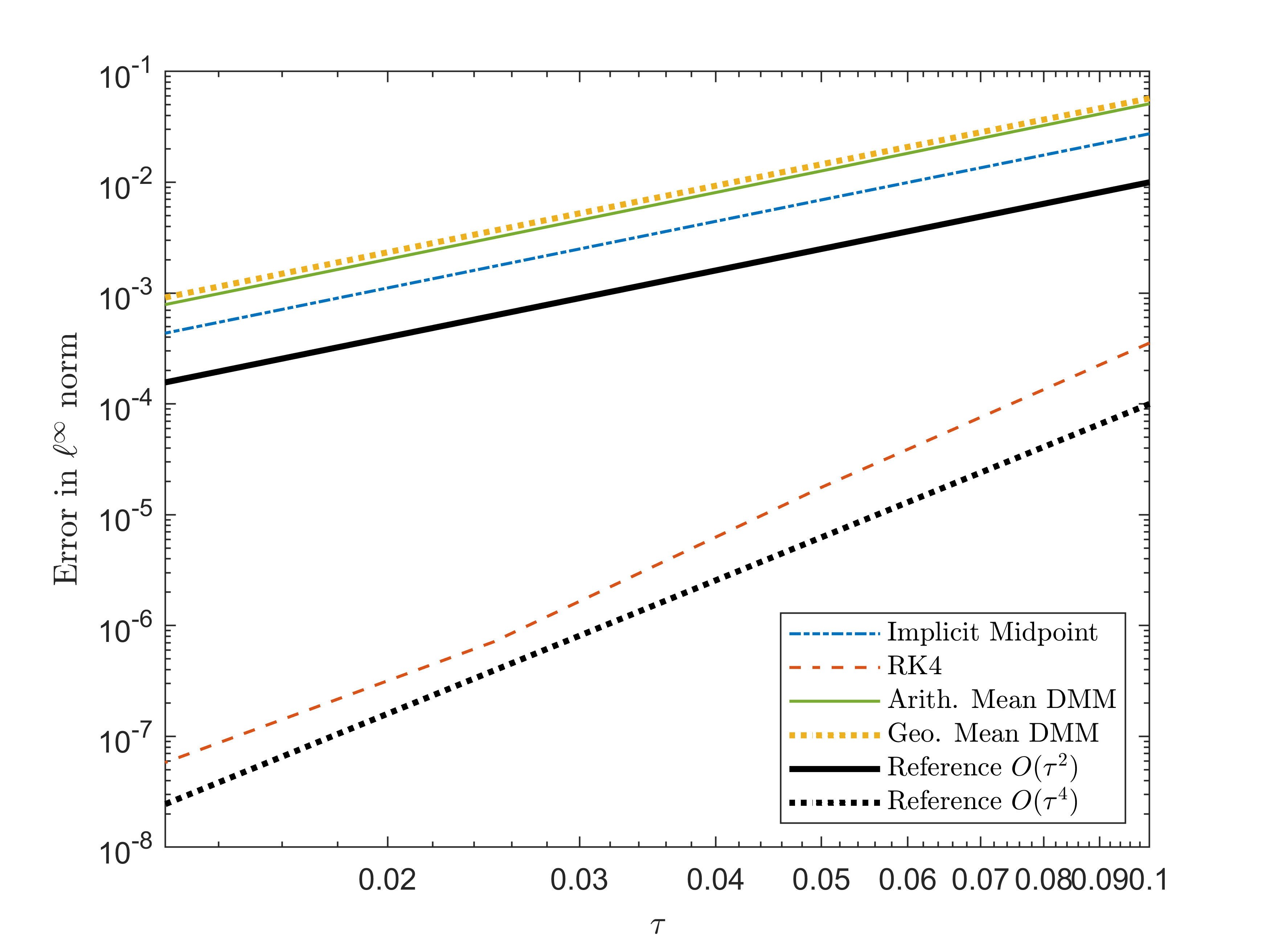}
  \caption{Convergence plot for all methods.}
  \label{LV 3 Species--non degenerate--Convergence}
\end{subfigure}
\hfil
\begin{subfigure}{0.49\textwidth}
  \centering
  \includegraphics[width=\linewidth]{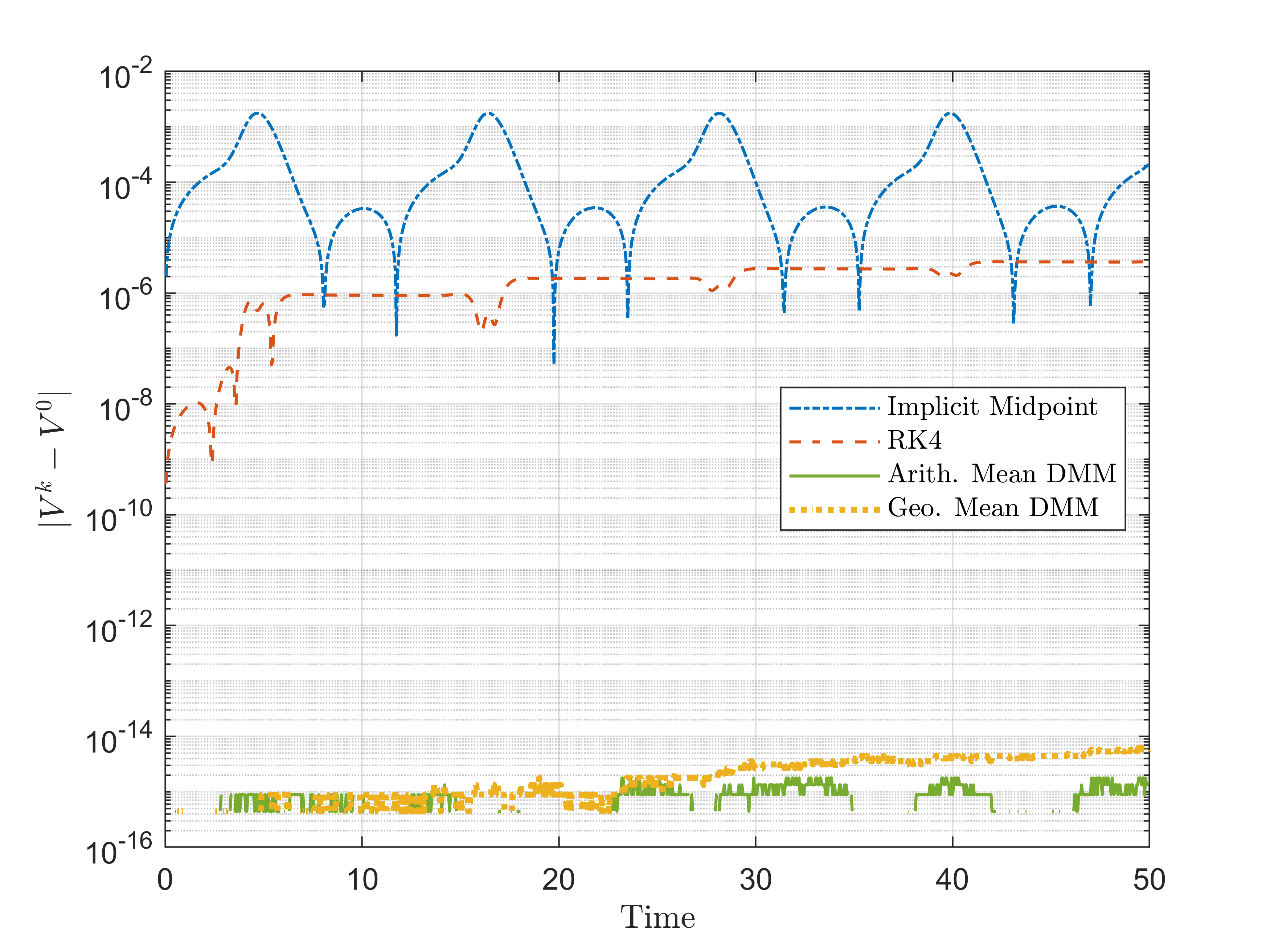}
  \caption{Error in conserved quantity $V$ versus time for all methods.}
  \label{LV 3 Species--non degenerate--CQError}
\end{subfigure}
\end{figure}

Figure \ref{LV 3 Species--non degenerate--trajectories} shows
the trajectories along the isosurface of $\mathcal{V}=\left\{ \bb x\in\mathbb{R}^{3}\left|V\left(\bb x\right)=V\left(\bb x_{0}\right)\right.\right\}.$ Qualitatively, we see that throughout a transient phase, all trajectories remain close to the isosurface $\mathcal{V}$. The dynamics then settles to a limit-cycle on the $y$--$z$ plane. While all trajectories are visually close to $\mathcal{V}$, only the trajectories computed by the DMM schemes are within machine-precision from $\mathcal{V}$. In contrast, trajectories of the Midpoint and fourth-order Runge--Kutta method are respectively $\approx 10^{-3}$ and $\approx 10^{-6}$ away from $\mathcal{V}$, even though it is difficult to discern visually from Figure~\ref{LV 3 Species--non degenerate--trajectories}.

\begin{figure}[!ht]
\centering
\includegraphics[width=28pc]{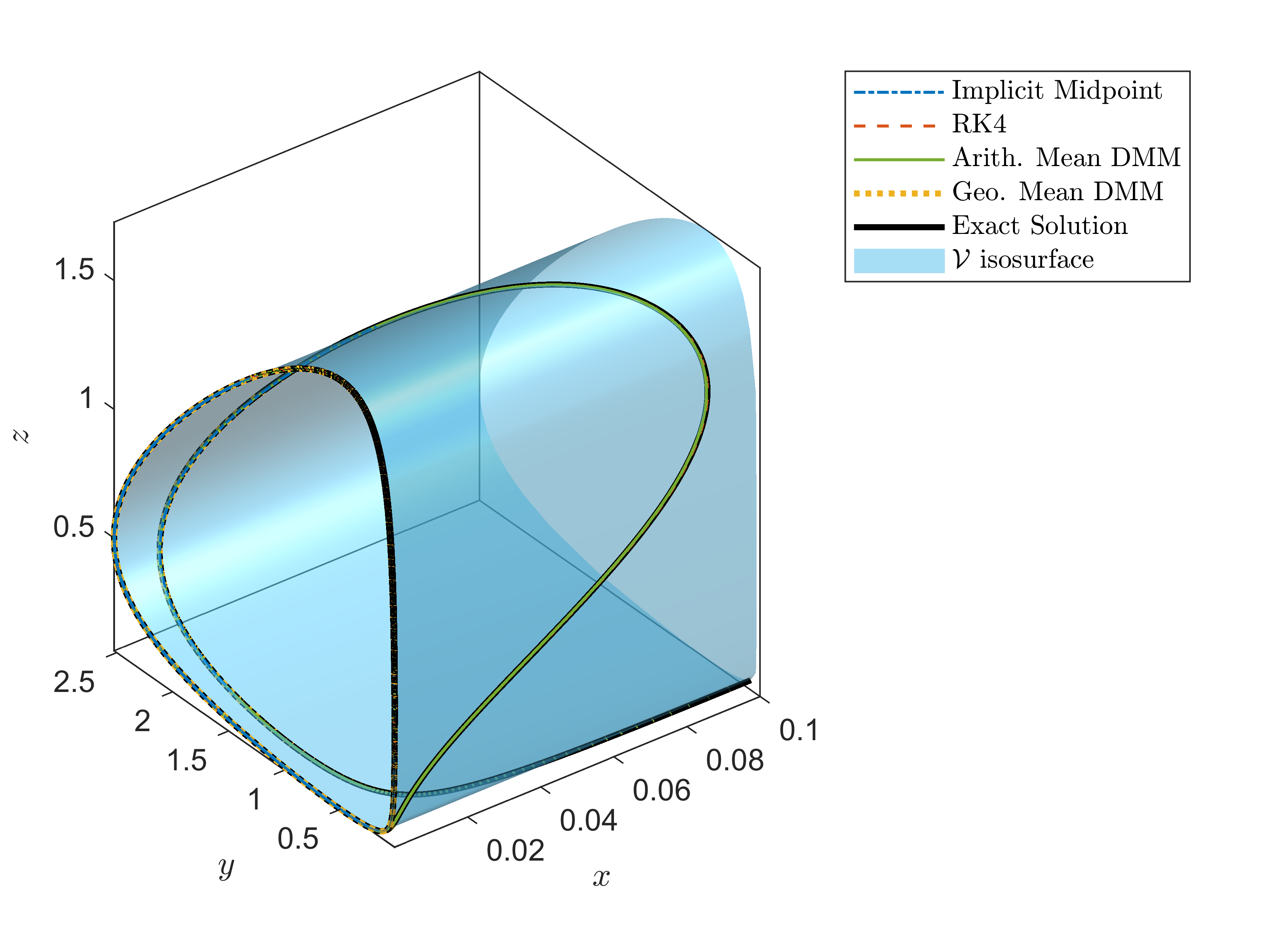}\caption{Trajectories of all methods for the three--species Lotka--Volterra system along $\mathcal{V}.$}
\label{LV 3 Species--non degenerate--trajectories}
\end{figure}

\subsection{Many--body problem: Gravitational potential}

Next we consider the standard ten--body solar system model. We compare the Greenspan--Labudde scheme, or equivalently the derived DMM scheme, with the St\"{o}rmer--Verlet and Midpoint schemes, which are two popular symplectic methods for many--body problems. The initial conditions were obtained from planetary data in \cite{FWBAE14}. We simulate the system to a final time $T=2\cdot10^{7}$ days ($\sim55,000$ years) using $N=4\cdot10^{6}$ time steps, corresponding to a uniform time step size of $\tau=5$ days. For the initial guess of the fixed point iterations, we have used a perturbation of the discrete solution from the previous time step in order to avoid singularities that can arise in evaluating divided differences.

Table~\ref{Table_CQ_Solar10} shows the error in $\ell^{\infty}$ norm of the ten conserved
quantities for the methods considered. While all three methods
perform quite similarly, the Greenspan--Labudde or DMM scheme is the only one which preserves the Hamiltonian up to machine precision. We also note that due to the time-dependent nature of the initial center of mass $\bb C(t,\bb q, \bb p)$, the error for this conserved quantity is not up to machine precision for the Greenspan--Labudde/DMM scheme. Specifically, since $\bb C$ is growing linearly in time, the discrete counterparts, which are zero up to machine precision, are being multiplied by a linear factor in time\footnote{For more details, see the calculations on the time-dependent terms of DMM in \eqref{eq:manybody_appendix_calc} of \ref{app:manybody}.}, resulting in $\sim10^{-11}$ error for $\bb C$.

\begin{table}[!ht]
\centering
\begin{tabular}{|c|c|c|c|}
\hline 
Method & Midpoint & St\"{o}rmer--Verlet & Greenspan--Labudde/DMM\tabularnewline
\hline 
\hline 
$\text{Error}[H(\bb q,\bb p)]$ & $7.03\cdot10^{-12}$ & $1.72\cdot10^{-12}$ & $2.03\cdot10^{-17}$\tabularnewline
\hline 
$\text{Error}[P_x(\bb p)]$ & $1.04\cdot10^{-18}$ & $1.44\cdot10^{-18}$ & $1.26\cdot10^{-18}$\tabularnewline
\hline 
$\text{Error}[P_y(\bb p)]$ & $8.84\cdot10^{-19}$ & $1.36\cdot10^{-18}$ & $9.44\cdot10^{-19}$\tabularnewline
\hline 
$\text{Error}[P_z(\bb p)]$ & $2.45\cdot10^{-19}$ & $9.41\cdot10^{-19}$ & $2.12\cdot10^{-19}$\tabularnewline
\hline 
$\text{Error}[L_x(\bb q,\bb p)]$ & $2.14\cdot10^{-17}$ & $4.15\cdot10^{-18}$ & $7.44\cdot10^{-18}$\tabularnewline
\hline 
$\text{Error}[L_y(\bb q,\bb p)]$ & $3.03\cdot10^{-16}$ & $4.94\cdot10^{-18}$ & $1.23\cdot10^{-16}$\tabularnewline
\hline 
$\text{Error}[L_z(\bb q,\bb p)]$ & $5.41\cdot10^{-16}$ & $1.36\cdot10^{-17}$ & $2.24\cdot10^{-16}$\tabularnewline
\hline 
$\text{Error}[C_x(t, \bb q,\bb p)]$ & $1.07\cdot10^{-11}$ & $1.24\cdot10^{-11}$ & $1.50\cdot10^{-11}$\tabularnewline
\hline 
$\text{Error}[C_y(t, \bb q,\bb p)]$ & $8.89\cdot10^{-12}$ & $1.75\cdot10^{-11}$ & $1.11\cdot10^{-11}$\tabularnewline
\hline 
$\text{Error}[C_z(t, \bb q,\bb p)]$ & $3.51\cdot10^{-12}$ & $1.06\cdot10^{-11}$ & $4.56\cdot10^{-12}$\tabularnewline
\hline 
\end{tabular}
\caption{\label{Table_CQ_Solar10} Error in conserved quantities for the ten--body solar
system model.}
\end{table}

Figure \ref{CQ_vs_time_All} shows the error in time for the ten conserved quantities of the three methods considered. Furthermore, Figure~\ref{Trajectories_Solar10} shows that the trajectories of all three methods are in good agreement qualitatively.
\newpage
\begin{figure}[ht!]
\centering
\includegraphics[width=40pc]{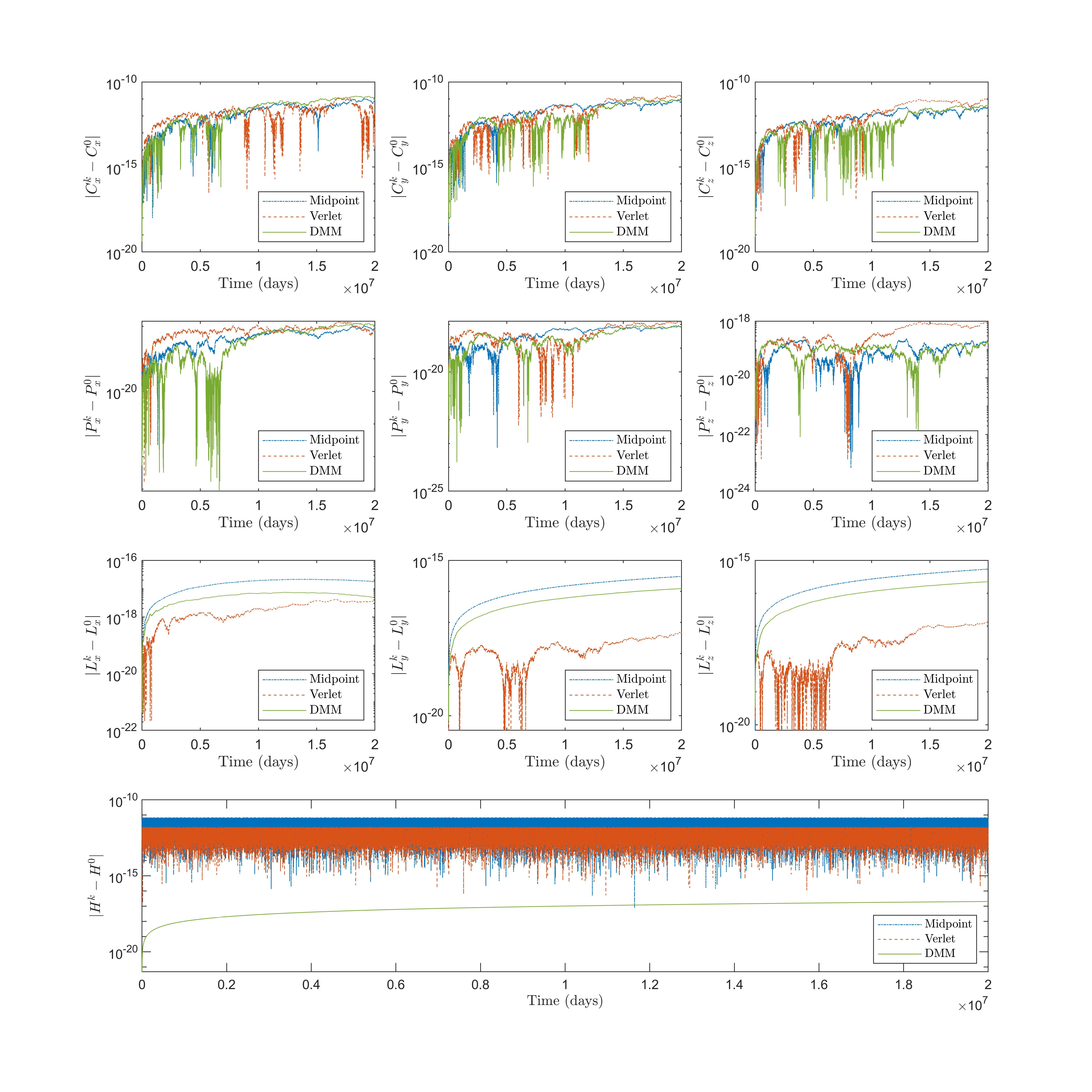}
\caption{\label{CQ_vs_time_All} Error in conserved quantities versus time for the ten--body solar system using the Midpoint method, St\"{o}rmer--Verlet method and Greenspan--Labudde/DMM scheme, plotted every 500 days.}
\end{figure}
\newpage
\begin{figure}[ht!]
\centering
\includegraphics[width=35pc]{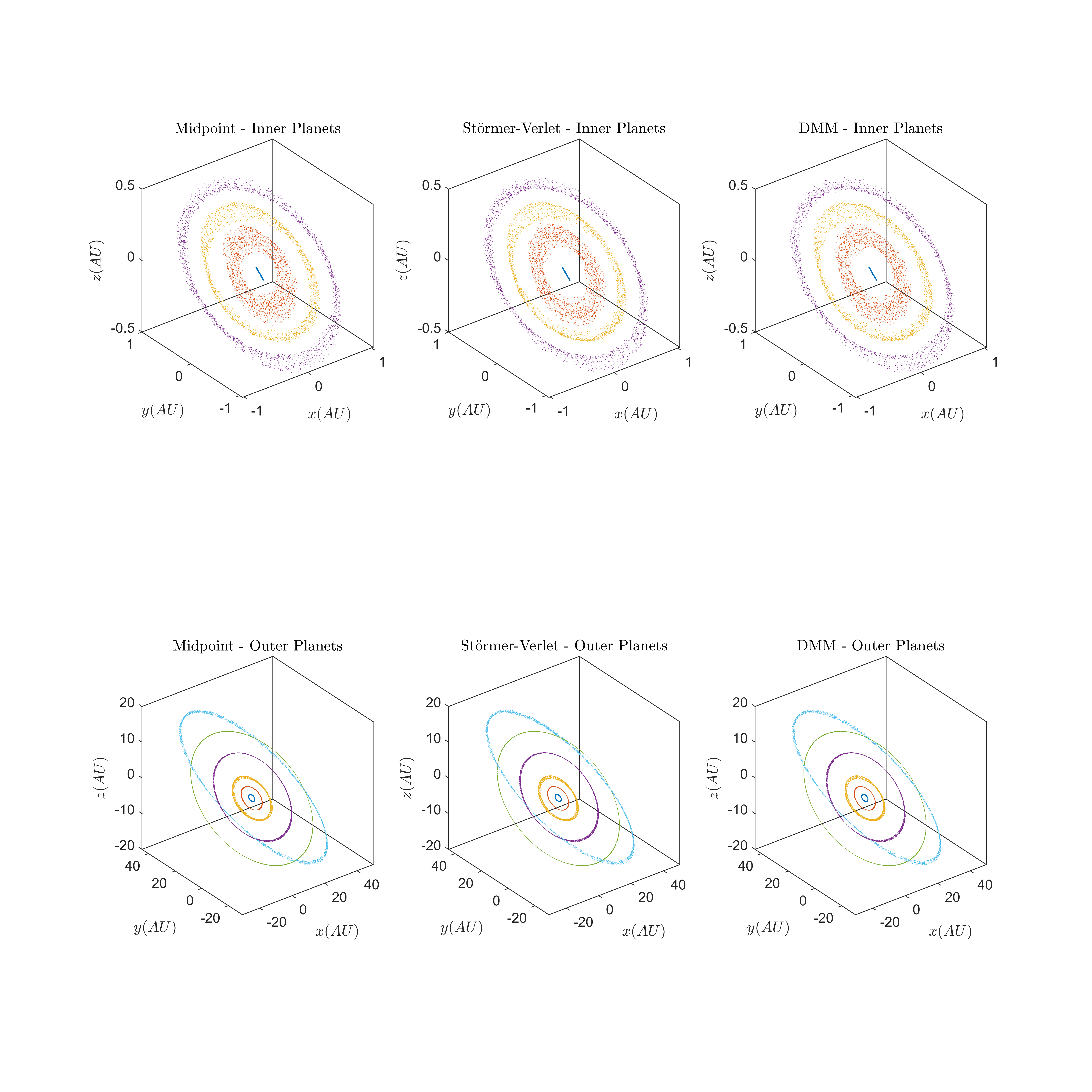}
\caption{\label{Trajectories_Solar10} Trajectories for all three methods, plotted every 2500 days.}
\end{figure}

\subsection{Many--body problem: Lennard--Jones Potential}

Next we consider an example from molecular dynamics using the Lennard--Jones Potential. For simplicity, we have chosen to look at a frozen Argon crystal model in 2D. This is to avoid complications which arises when restricting molecules to a finite domain, such as to truncate the potential to a finite radial length and to handle periodic boundary conditions in a conservative manner. We mentioned that \cite{SR21} recently resolved these difficulties in the case of general pairwise and three--body interaction potentials.

As with the gravitational potential example, we compare the derived DMM scheme, with the St\"{o}rmer--Verlet and the Midpoint method. For this problem, we used the initial conditions and parameters provided by \cite[Chapter I.4]{hair06Ay}. The units  are chosen to be nanoseconds $\left[ns\right]$ for
time, nanometers $\left[nm\right]$ for length, and $\left[kg\right]$
for mass. The number of atoms is set to $n=7$ and the mass of each
atom is uniformly set to $m_{i}=66.34\cdot10^{-27}\:\left[kg\right].$ We further set $\sigma=0.341\:\left[nm\right],$ and $\epsilon=119.8\,k_{B}~\left[J\right],$ where $k_{B}=1.380658\cdot10^{-23}\:\left[J\cdot K^{-1}\right]$ is
Boltzmann's constant. In the actual simulation, we have rescaled the equations
by $k_{B}$ in order to mitigate round-off errors. As a result, the conserved
quantities presented are in these rescaled units. We simulate the system to a final time $T=0.2~[ns]$ and $N=4\cdot10^{4}$ time steps, corresponding to a time step size of $\tau=50$ femtoseconds. Similar to the gravitational potential example, we used a perturbation of the discrete solution from the previous time step for the initial guess in the fixed point iterations.

Since the problem is two-dimensional, only six conserved quantities
are relevant and their respective $\ell^{\infty}$ norm errors are shown in Table \ref{Table_CQ_Argon}. Similarly to previous examples, we see that the derived DMM scheme is the only one preserving all conserved quantities. We note that there are larger accumulation of round-off errors for the Hamiltonian due to cancellation errors from the more complex divided difference expression \eqref{eqn:LJ_dividedDiff} arising from the Lennard--Jones potential. Furthermore, we observed good qualitative agreement of trajectories with other methods in Figure \ref{Trajectories_Argon} and favorable conservative properties of DMM in Figure \ref{CQ_vs_time_Argon}. 

\begin{table}[ht!]
\centering
\begin{tabular}{|c|c|c|c|}
\hline 
 & Midpoint & St\"{o}rmer-Verlet & DMM\tabularnewline
\hline 
\hline 
$\text{Error}[H(\bb q,\bb p)]$ & $3.17\cdot10^{-2}$ & $4.60\cdot10^{-2}$ & $7.84\cdot10^{-11}$\tabularnewline
\hline 
$\text{Error}[P_x(\bb q,\bb p)]$ & $9.66\cdot10^{-15}$ & $1.95\cdot10^{-14}$ & $8.16\cdot10^{-15}$\tabularnewline
\hline 
$\text{Error}[P_y(\bb q,\bb p)]$ & $5.50\cdot10^{-15}$ & $3.08\cdot10^{-14}$ & $5.47\cdot10^{-15}$\tabularnewline
\hline 
$\text{Error}[L_z(\bb q,\bb p)]$ & $4.41\cdot10^{-15}$ & $4.86\cdot10^{-15}$ & $3.36\cdot10^{-15}$\tabularnewline
\hline 
$\text{Error}[C_x(t,\bb q,\bb p)]$ & $3.53\cdot10^{-14}$ & $9.37\cdot10^{-14}$ & $2.10\cdot10^{-14}$\tabularnewline
\hline 
$\text{Error}[C_y(t,\bb q,\bb p)]$ & $3.41\cdot10^{-14}$ & $1.11\cdot10^{-13}$ & $2.67\cdot10^{-14}$\tabularnewline
\hline 
\end{tabular}
\caption{\label{Table_CQ_Argon} Error in conserved quantities for the frozen Argon crystal model.}
\end{table}

\begin{figure}[ht!]
\centering
\includegraphics[width=21pc]{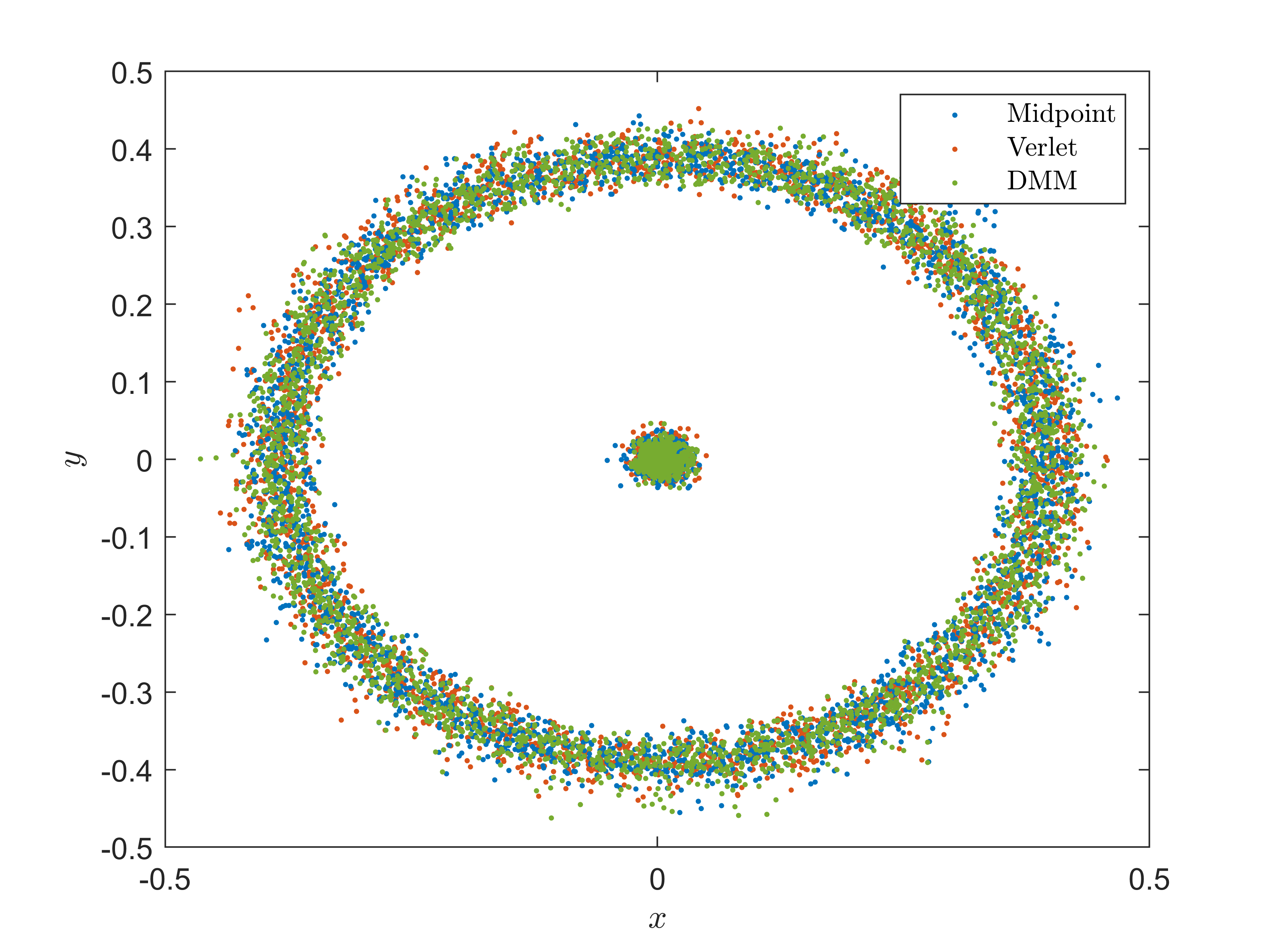}
\caption{\label{Trajectories_Argon}Trajectories of the seven Argon atoms for
all three methods.}
\end{figure}

\begin{figure}[ht!]
\centering
\includegraphics[width=40pc]{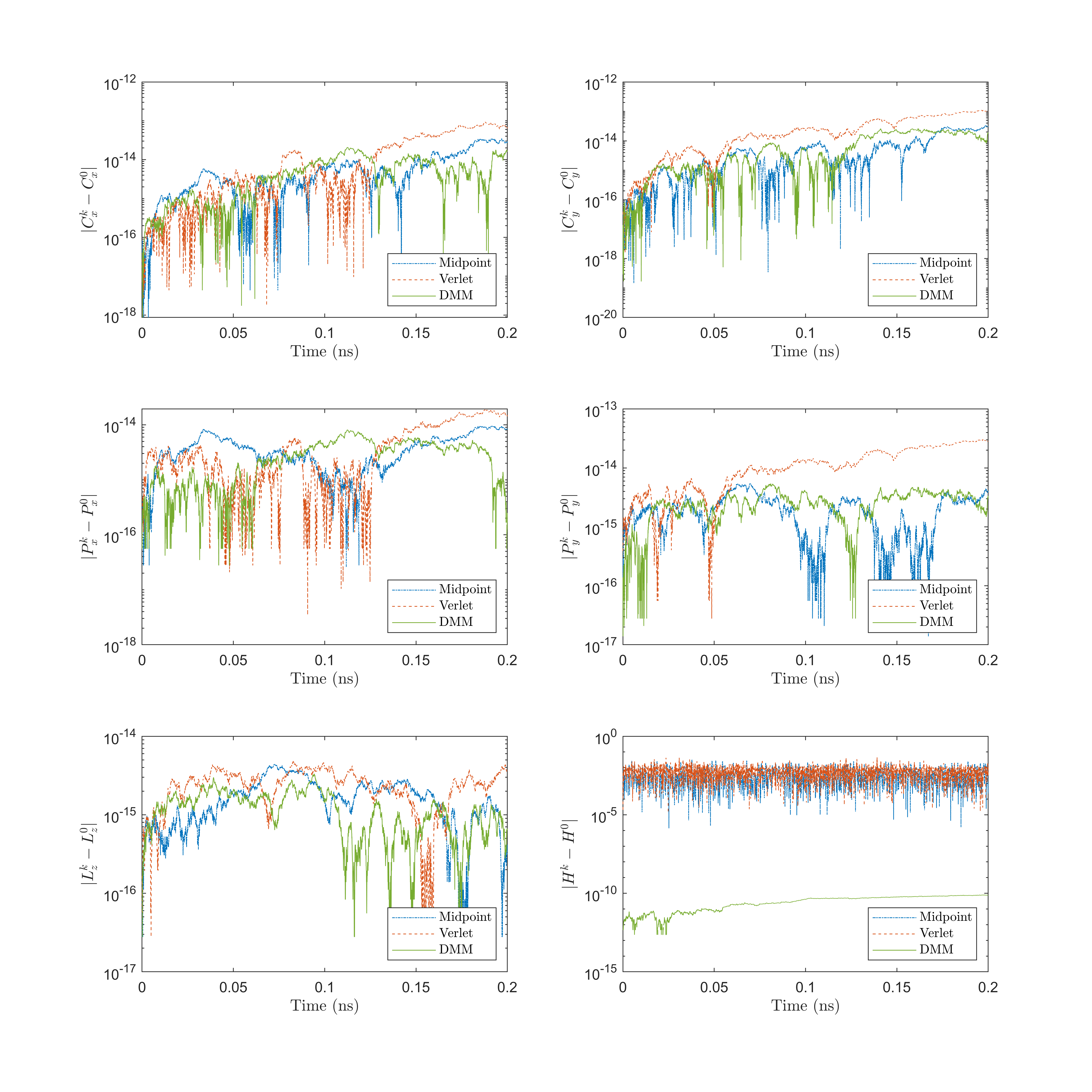}
\caption{\label{CQ_vs_time_Argon}Error in conserved quantities versus time for the frozen
Argon crystal model.}
\end{figure}

\subsection{Point vortex problem: Planar case}

Next, we investigate and compare numerical results of DMM with the Midpoint method and the standard explicit fourth-order Runge--Kutta method.

The test consists of evolving $n=1000$ randomly distributed vortices. The initial locations were sampled from a uniform distribution on $\left[-5,5\right]^{2}$ and post-processed to ensure that no two vortices were closer than a minimum distance of $10/n=10^{-2}$. Furthermore, the vorticity strengths $\Gamma_{i}$ were sampled uniformly from $[-1,1]/n$. We run the simulation to a final time of $T=100$ for a total of $N=1000$ time steps, corresponding to a time step size $\tau=10^{-1}$. As with the many--body examples, we have used a perturbation of the discrete solution from the previous time step for the initial guess of the fixed point iterations.

\begin{figure}[ht!]
\includegraphics[width=35pc]{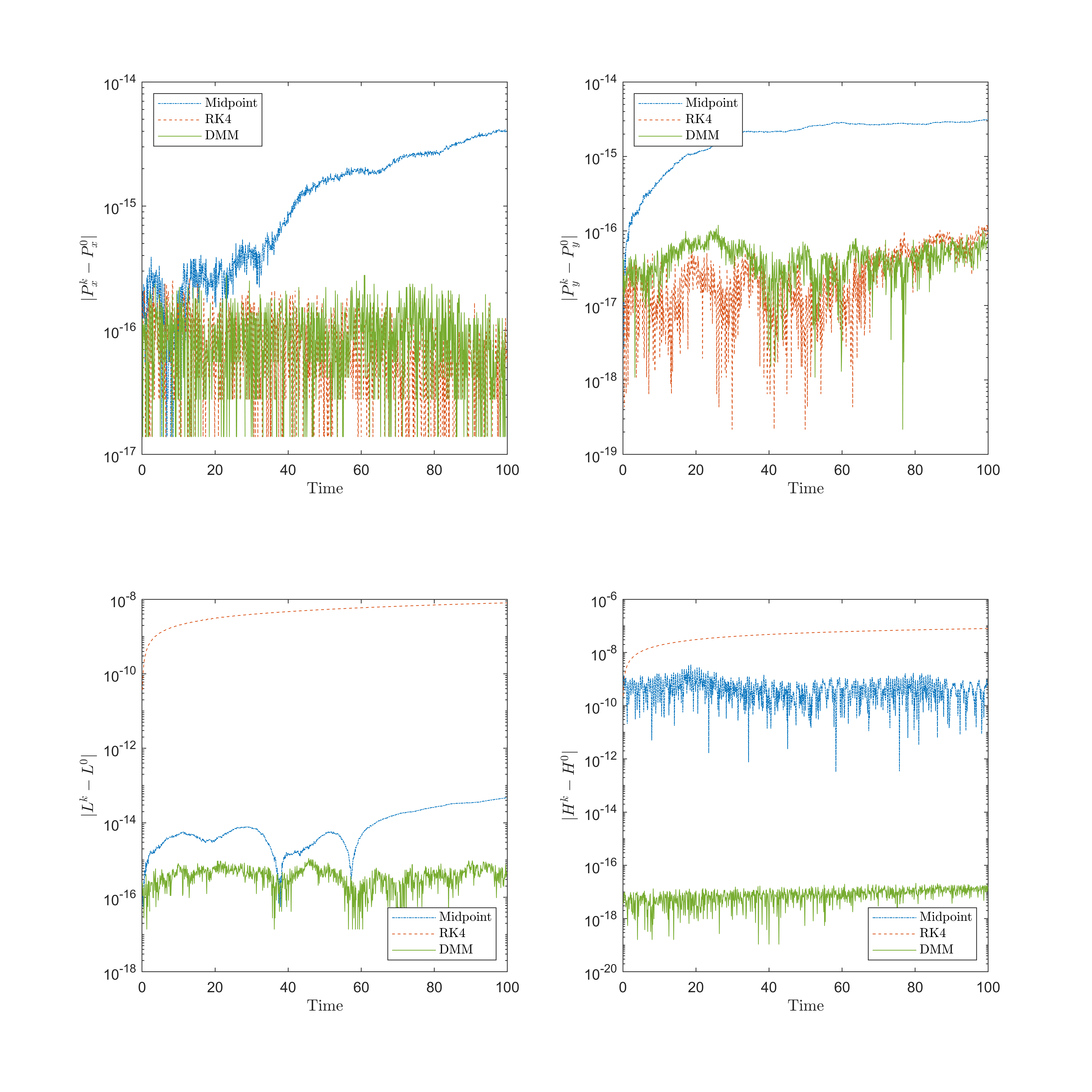}
\centering
\caption{\label{fig:Conserved-Quantities-forPV_Plane_All} Error in conserved quantities versus time for the planar point vortex problem.}
\end{figure}

Figure \ref{fig:Conserved-Quantities-forPV_Plane_All} shows the error in conserved quantities over time and Figure \ref{fig:Trajectories-for-all_PV_N=00003D1000_ALL} shows the trajectories of all $n=1000$ vortices for all three methods considered. We observe that most vortices have qualitatively similar trajectories for all three methods. However, there are a few vortices, with more subtle interactions, showing rather different trajectories. Specifically, we highlight these differences in Figure \ref{fig:Trajectories-for-all_PV_N=00003D1000_ALL} and by zooming in on their dynamics in Figure \ref{fig:Zoom-on-some_PV_N=00003D1000}. 
\begin{figure}[ht!]
\includegraphics[width=36pc]{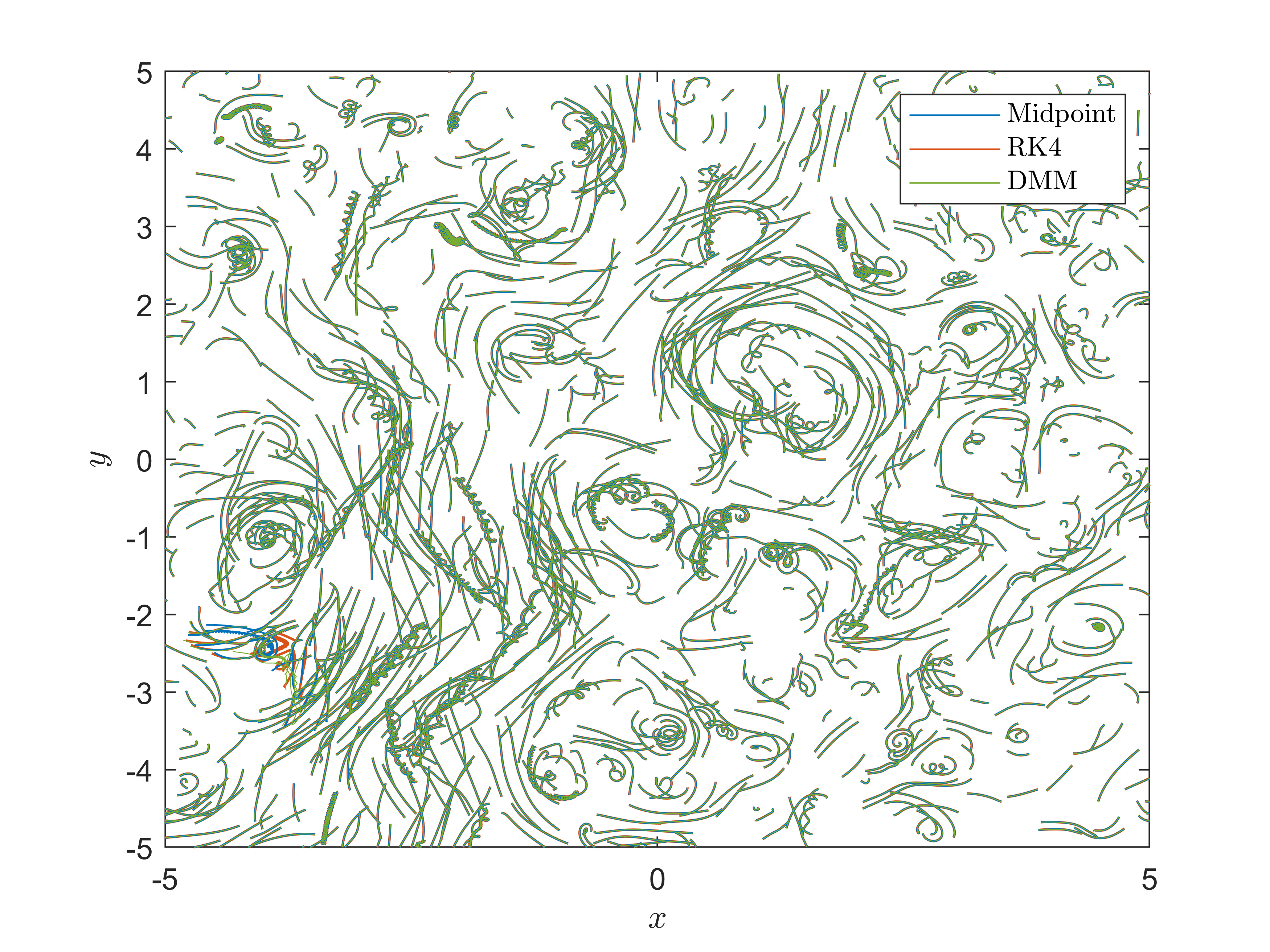}
\caption{\label{fig:Trajectories-for-all_PV_N=00003D1000_ALL}Trajectories for all three methods for the planar point vortex problem with $n=1000$ vortices.}
\end{figure}
\begin{figure}[ht!]
\begin{tabular}{cc}
\includegraphics[width=20pc]{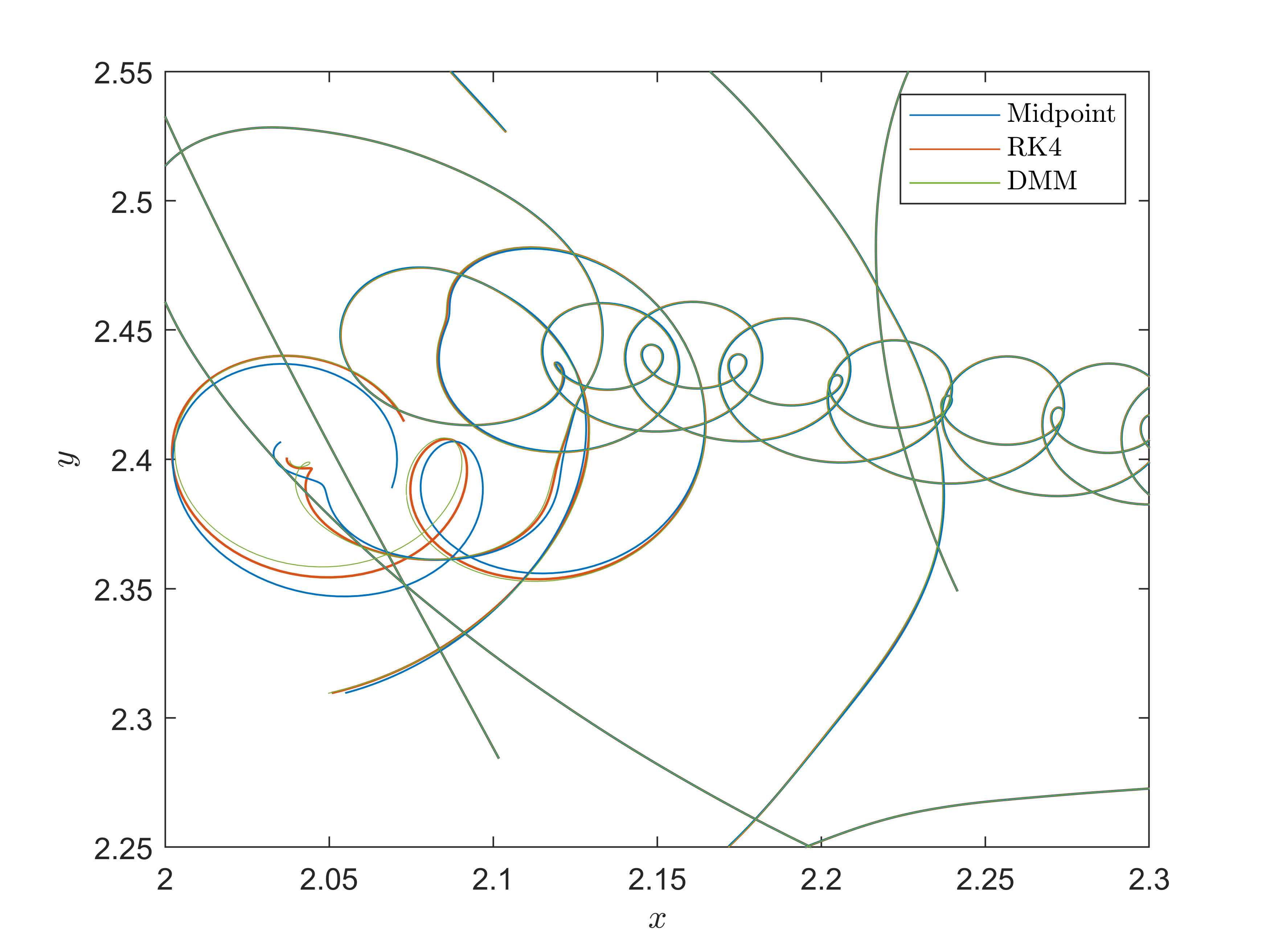} & \includegraphics[width=20pc]{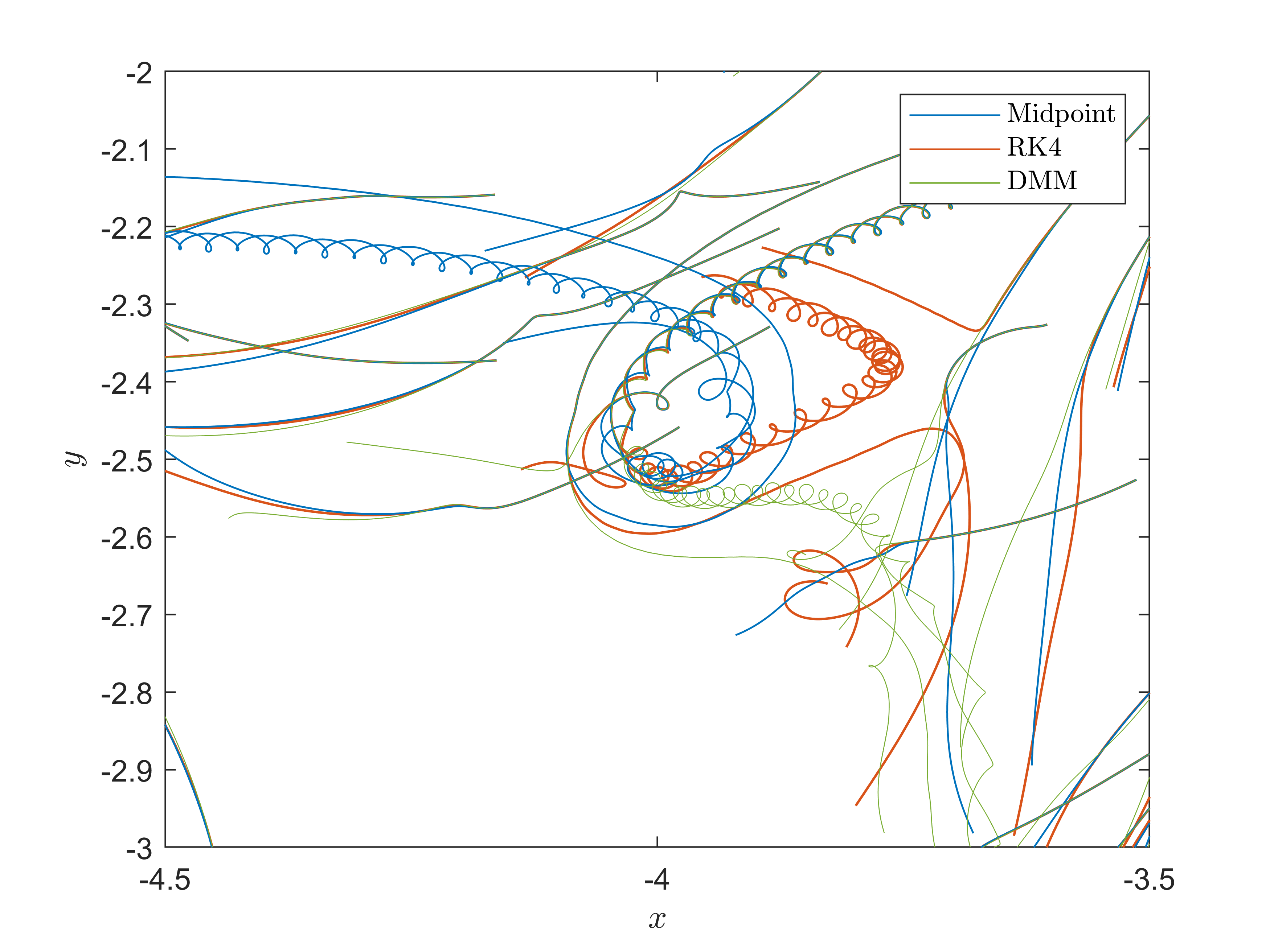}\tabularnewline
\end{tabular}
\caption{\label{fig:Zoom-on-some_PV_N=00003D1000}Closeup of trajectory differences for all three methods for the planar point vortex problem with $n=1000$ vortices.}
\end{figure}While the final time here is relatively short, we emphasize that for long integrations of these point vortex equations, these small differences in the trajectories will likely amplify, with the DMM being the only method preserving the energy up to machine precision.

Table \ref{tab:Conserved-quantities-for PV_Plane_N=00003D1000} shows the error in $\ell^{\infty}$ norm for all four conserved quantities.
\begin{table}[!ht]
\centering
\begin{tabular}{|c|c|c|c|}
\hline 
 & Midpoint & RK4 & DMM\tabularnewline
\hline 
\hline 
$\text{Error}[P_{x}(\bb x)]$ & $4.15\cdot10^{-15}$ & $2.64\cdot10^{-16}$ & $2.78\cdot10^{-16}$\tabularnewline
\hline 
$\text{Error}[P_{y}(\bb y)]$ & $3.16\cdot10^{-15}$ & $1.22\cdot10^{-16}$ & $1.20\cdot10^{-16}$\tabularnewline
\hline 
$\text{Error}[L(\bb x, \bb y)]$ & $4.69\cdot10^{-14}$ & $8.02\cdot10^{-9}$ & $1.11\cdot10^{-15}$\tabularnewline
\hline 
$\text{Error}[H(\bb x, \bb y)]$ & $3.49\cdot10^{-9}$ & $7.96\cdot10^{-8}$ & $2.14\cdot10^{-17}$\tabularnewline
\hline 
\end{tabular}
\caption{\label{tab:Conserved-quantities-for PV_Plane_N=00003D1000}Error in conserved
quantities for $n=1000$ point vortices on the plane.}
\end{table}
This table shows that all methods preserve the linear momentum up to machine precision, with the Midpoint method additionally preserving the angular momentum\footnote{This is expected as Midpoint method preserves all quadratic invariants, see \cite[Chapter IV.2]{hair06Ay}}. The DMM is the only method that also preserves the Hamiltonian up to machine precision. This is particularly noteworthy as the DMM method is only a second--order method, in comparison to the fourth-order Runge--Kutta method.

\subsection{Point vortex problem: Spherical case}

In this final example, similar to the previous section, we investigate and compare numerical results of DMM with the Midpoint method and the standard explicit fourth-order Runge--Kutta method applied to the point vortex equations on the unit sphere.

As with the planar case, the test consists of evolving $n=1000$ randomly distributed vortices. The initial locations were sampled from a uniform distribution
on the unit sphere according to the procedure proposed in~\cite{M72}. As in the planar case, we filtered the initial sampled locations to ensure that no two vortices are closer than a minimum distance of $4\pi/n\approx1.3\cdot10^{-2}$. Furthermore, the vorticity strengths $\Gamma_{i}$ were sampled uniformly from $[-1,1]/n$. We run the simulation to a final time of $T=100$, with $N=1000$ time steps, corresponding to a time step size of $\tau=10^{-1}$. As before, we used a perturbation of the discrete solution from the previous time step for the initial guess in the fixed point iterations.

We show in Table \ref{tab:Conserved-quantities-for PV_Sphere_N=00003D1000} the error in  $\ell^{\infty}$ norm for all four conserved quantities.

\begin{table}[!ht]
\centering
\begin{tabular}{|c|c|c|c|}
\hline 
 & Midpoint & RK4 & DMM\tabularnewline
\hline 
\hline 
$\text{Error}[P_{x}(\bb x)]$ & $5.03\cdot10^{-17}$ & $3.99\cdot10^{-17}$ & $4.16\cdot10^{-17}$\tabularnewline
\hline 
$\text{Error}[P_{y}(\bb x)]$ & $3.95\cdot10^{-17}$ & $4.12\cdot10^{-17}$ & $4.47\cdot10^{-17}$\tabularnewline
\hline 
$\text{Error}[P_{z}(\bb x)]$ & $3.47\cdot10^{-17}$ & $4.86\cdot10^{-17}$ & $5.38\cdot10^{-17}$\tabularnewline
\hline 
$\text{Error}[H(\bb x)]$ & $5.44\cdot10^{-10}$ & $1.55\cdot10^{-11}$ & $3.73\cdot10^{-18}$\tabularnewline
\hline 
\end{tabular}
\caption{\label{tab:Conserved-quantities-for PV_Sphere_N=00003D1000}Error in conserved
quantities versus time for $n=1000$ point vortices on the sphere.}
\end{table}
Figure~\ref{fig:Conserved-Quantities-forPV_Sphere_all} depicts the error in conserved quantities over time for all three methods and Figure~\ref{fig:Trajectories-for-all_PV_Sphere_N=00003D1000_ALL} shows the trajectories of all $1000$ vortices. As with the planar case, we observe that most trajectories of the three methods are qualitatively indistinguishable, although certain vortices do exhibit increasingly diverging trajectories over the short integration interval. This is highlighted by zooming in on specific areas presented in Figure \ref{fig:Zoom-on-some_PV_sphere_N=00003D1000}.

\newpage
\begin{figure}[ht!]
\includegraphics[width=40pc]{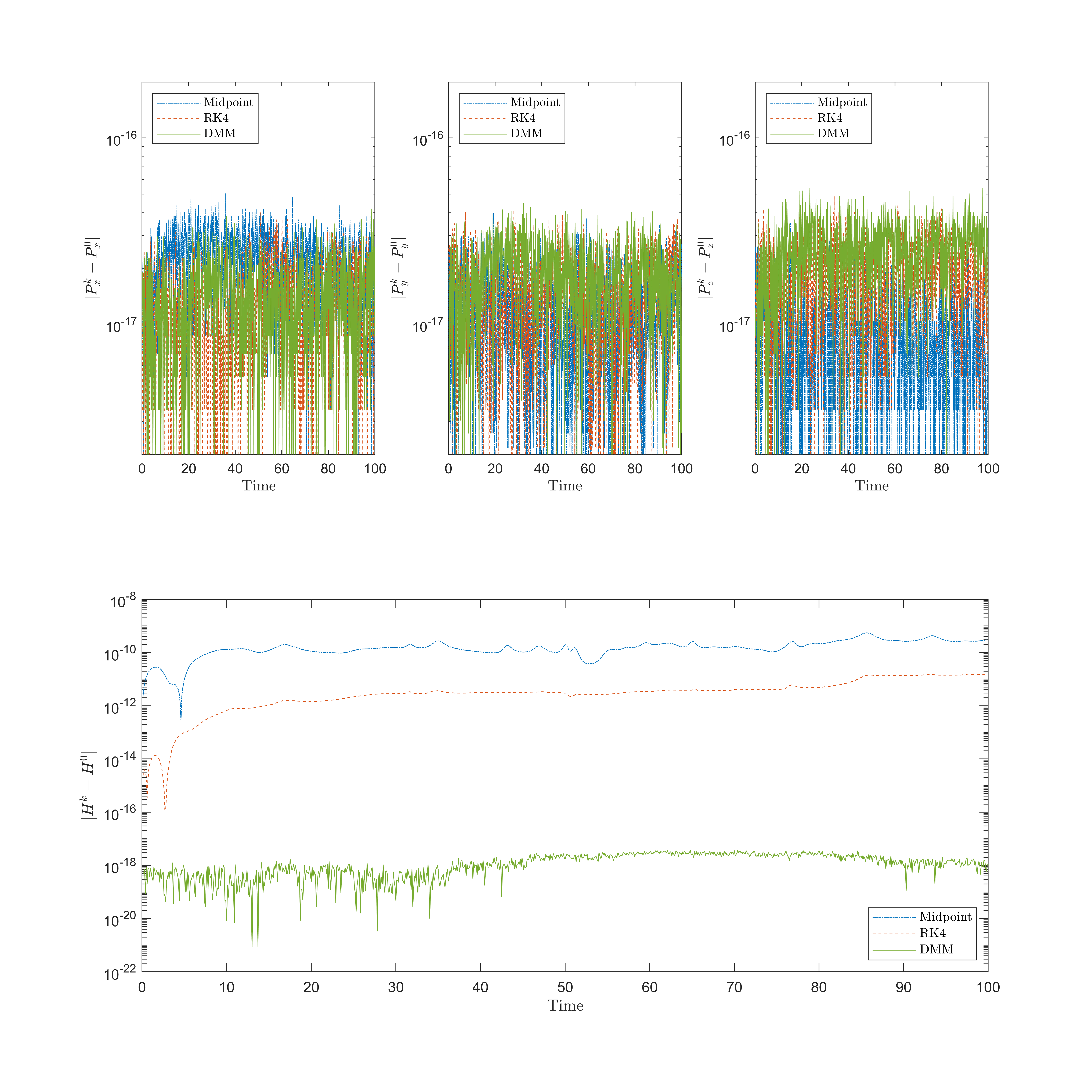}

\caption{\label{fig:Conserved-Quantities-forPV_Sphere_all}Error in conserved quantities for the point vortex problem on the unit sphere with $n=1000$ vortices.}
\end{figure}

\newpage
\begin{figure}[ht!]
\includegraphics[width=40pc]{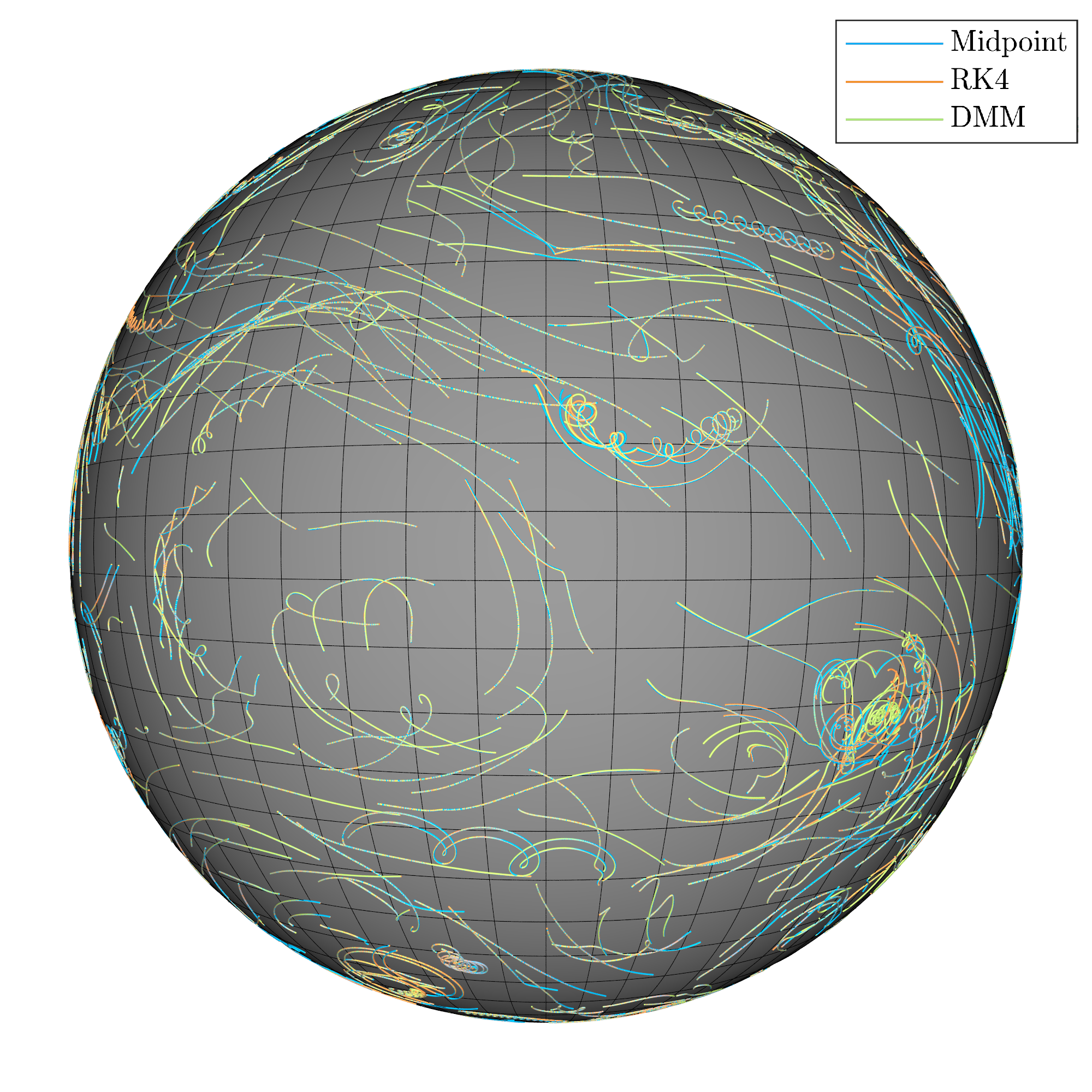}
\caption{\label{fig:Trajectories-for-all_PV_Sphere_N=00003D1000_ALL}Trajectories for all three methods for the point vortex problem on the unit sphere with $n=1000$ vortices.}
\end{figure}

\newpage
\begin{figure}[ht!]
\begin{tabular}{cc}
\includegraphics[width=18pc]{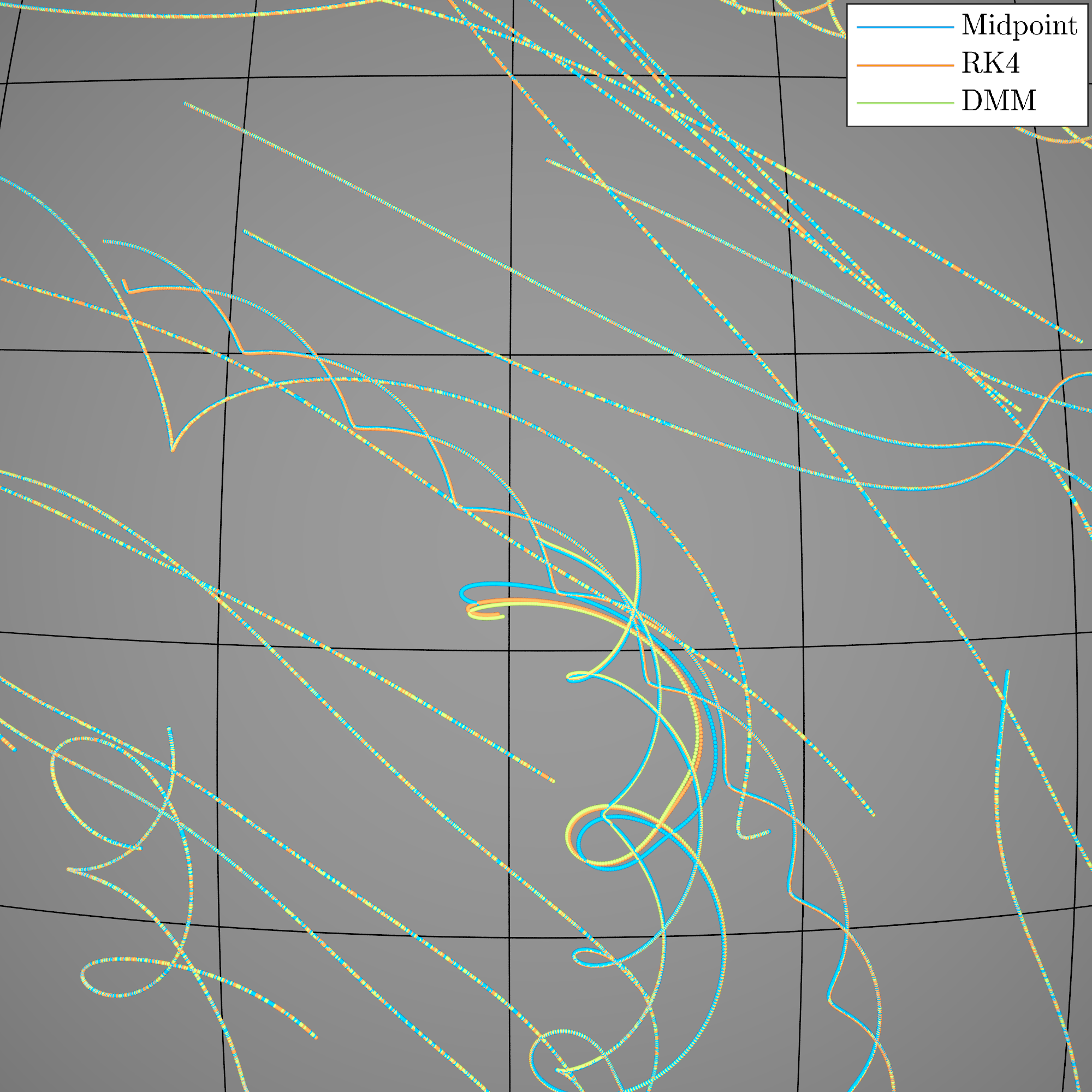} & \includegraphics[width=18pc]{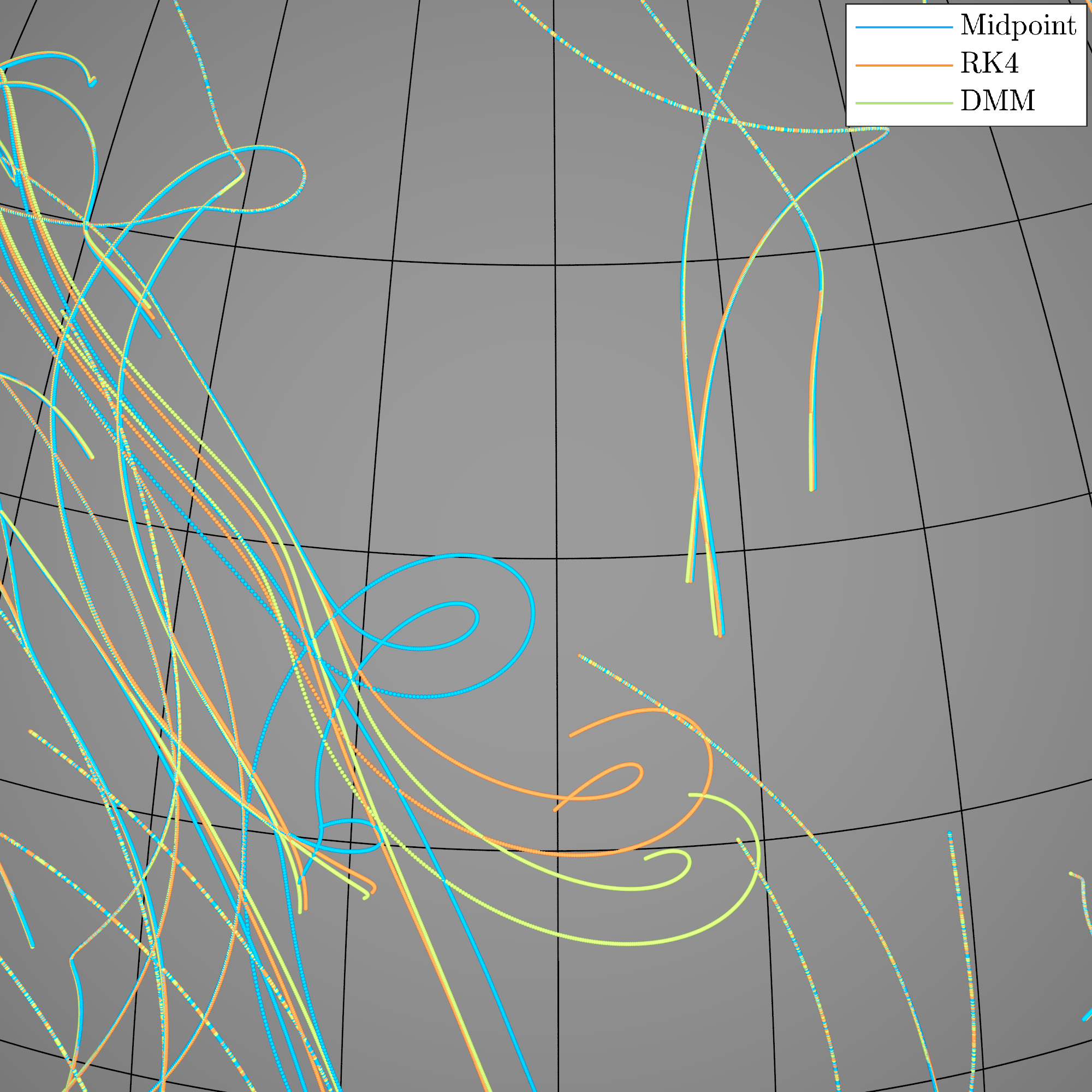}\tabularnewline
\end{tabular}
\caption{\label{fig:Zoom-on-some_PV_sphere_N=00003D1000}Closeup on some diverging trajectories for all three methods for the point vortex problem on the unit sphere with $n=1000$ vortices.}
\end{figure}

\section{Conclusions}\label{sec:ConclusionsManyBody}
In this paper, we have constructed several conservative numerical schemes for important mathematical models arising from a wide variety of different fields. The framework used to derive such numerical schemes is the Discrete Multiplier Method, or DMM, originally developed in~\cite{wan16a}. DMM is not only suitable for low--dimensional dynamical systems with multiple conserved quantities, as shown in \cite{wan16a}, but also as we showed here, it is applicable for constructing conservative schemes for many--body systems. As the DMM does not require additional geometric structures from the underlying dynamical system other than the presence of invariants themselves, this approach can potentially be applied to a wide variety of other dynamical systems arising in the mathematical sciences. 

With the derived conservative schemes for many--body systems, there are still practical drawbacks which we wish to improve in future work. The main drawback of the derived DMM schemes so far is that they are implicit, as it is not currently known if there are explicit DMM schemes for general dynamical systems. For large many-body Hamiltonian systems, explicit methods such as St\"{o}rmer-Verlet method and higher--order symplectic splitting methods are often preferred due to their lower computational costs when comparing accuracy versus number of force evaluations. To reduce the computational costs of the implicit DMM schemes, one can use high--order DMM schemes and specific quasi-Newton methods aimed at improving the efficiency of solving nonlinear equations arising from DMM.

Despite these current practical limitations, we emphasize that our numerical results, specifically in the point vortex examples, indicate that a higher--order method, while more accurate at approximating the solution, does not necessarily imply it is more accurate at preserving conserved quantities when compared to a lower--order conservative method.

\section*{Acknowledgements}

ATSW was partially supported by the CRM and the NSERC Discovery Grant program. AB was supported by the Canada Research Chairs program, the InnovateNL Leverage R{\&}D program and the NSERC Discovery Grant program. JCN was supported by the NSERC Discovery Grant program. Additionally, JCN would like to thank Prof. Wenjun Ying of the Natural Science Institute of the Shanghai Jiao Tong University for his kind hosting. The environment provided by the NSI was invaluable during the early stages of this project.

\bibliography{refs}

\begin{thebibliography}{10}
\bibitem{deVo56a}
R.~de~Vogelaere, {}methods of integration which preserve the contact
  transformation property of the hamiltonian equations, Tech. rep.
\end{thebibliography}

\appendix

\section{Verification of conservative and symmetric properties}

\label{sec:Appendix}

In this Appendix, we present some details of the derivations of the conservative schemes and verification of symmetric property used throughout this paper.

Due to the common appearance of the function $g(z) = \dfrac{\log z}{z-1}$, we first show the following elementary lemma useful in proving the symmetric property of our schemes.

\begin{lem}Let $g(z) = \dfrac{\log z}{z-1}$ and $h$ be any scalar function which maps from the phase space into real numbers excluding the zero, i.e. $h:\overbrace{\mathbb{R}^d\times \cdots \times \mathbb{R}^d}^{n \text{ copies}}\rightarrow \mathbb{R}\backslash \{0\}$. Then,
\begin{align}
\frac{1}{h(\bb x^k)}g\left(\frac{h(\bb x^{k+1})}{h(\bb x^k)}\right) = \frac{1}{h(\bb x^{k+1})}g\left(\frac{h(\bb x^{k})}{h(\bb x^{k+1})}\right).
\end{align} \label{lem:sym}
\end{lem}
\begin{proof}
\[
\frac{1}{h(\bb x^k)}g\left(\frac{\bb x^{k+1}}{\bb x^k}\right) &= \frac{1}{h(\bb x^k)}\frac{\log\left(\dfrac{h(\bb x^{k+1})}{h(\bb x^k)}\right)}{\dfrac{h(\bb x^{k+1})}{h(\bb x^k)}-1} = \frac{\log\left(h(\bb x^{k+1})\right)-\log\left(h(\bb x^{k})\right)}{h(\bb x^{k+1})-h(\bb x^{k})} \\&= \frac{1}{h(\bb x^{k+1})}\frac{\log\left(\dfrac{h(\bb x^{k})}{h(\bb x^{k+1})}\right)}{\dfrac{h(\bb x^{k})}{h(\bb x^{k+1})}-1} = \frac{1}{h(\bb x^{k+1})}g\left(\frac{h(\bb x^{k})}{h(\bb x^{k+1})}\right).
\]
\end{proof}
In other words, $\dfrac{1}{h(\bb x^k)}g\left(\dfrac{h(\bb x^{k+1})}{h(\bb x^k)}\right)$ is symmetric under the permutation $k\leftrightarrow k+1$.

\subsection{Lotka--Volterra systems}
\label{app:LV}

We first show that $V(\bb x)$ is a conserved quantity of \eqref{LVsys} and the derived scheme is conservative and symmetric using Lemma \ref{lem:sym}.

Using \eqref{multCond1}, the associated multiplier $1\times n$ matrix is given by
\[
\Lambda(\bb x):=\begin{bmatrix} d_i\left(\dfrac{\xi_i}{x_i}-1\right)\end{bmatrix}_{1\leq i\leq n}
\]
which verifies $V(\bb x)$ is indeed a conserved quantity since,
\[
\Lambda(\bb x) \bb F(\bb x) &= \sum_{i=1}^n d_i\left(\dfrac{\xi_i}{x_i}-1\right)\dot{x}_i - \sum_{i,j=1}^n d_i\left(\xi_i-x_i\right)a_{ij}(x_j - \xi_j) \\
&= D_t V(\bb x)+\underbrace{(\bb x-\bb \xi)^T DA(\bb x-\bb \xi)}_{=0 \text{ since }DA=-A^TD} = D_t V(\bb x).
\]
We employ DMM to derive conservative schemes for \eqref{LVsys}. Since $V$ is a linear combination of single variable functions, for any permutation $\sigma\in S_n$ of $x_1,\dots, x_n$, the discrete multiplier matrix is given by,
\[
\Lambda^\tau(\bb x^{k+1},\bb x^k) &= \begin{bmatrix} d_i\dfrac{\Delta}{\Delta x_i}\left(\xi_i \log x_i-x_i\right)\end{bmatrix}_{1\leq i\leq n}^T =\begin{bmatrix} d_i\left( \dfrac{\xi_i}{x_i^k}g\left(\dfrac{x_i^{k+1}}{x_i^k}\right)\right)\end{bmatrix}_{1\leq i\leq n}^T,
\] with $g(z) = \dfrac{\log z}{z-1}$.
Next, to discretize the right hand side $\bb f$, we first rewrite
\[
\bb f(\bb  x) &= \begin{bmatrix} x_i \sum_{j=1}^n a_{ij}x_j\left(1-\dfrac{\xi_j}{x_j}\right)\end{bmatrix}_{1\leq i\leq n}.
\] Since $g(z) = 1 + \mathcal{O}(z-1)$, let us propose by consistency of $\bb f^\tau$ to $\bb f$ that
\[
\bb f^\tau(\bb x^{k+1},\bb x^k) &=\begin{bmatrix} x_i^\tau\displaystyle\sum_{j=1}^n a_{ij}x_j^\tau\left(1- \dfrac{\xi_j}{x_j^k}g\left(\dfrac{x_j^{k+1}}{x_j^k}\right)\right)\end{bmatrix}_{1\leq i\leq n},
\] where $x_i^\tau$ is any consistent discretization of $x_i$.
Indeed, \eqref{discMultCond2} is satisfied for this choice of $\bb f^\tau$,
\[
\Lambda^\tau(\bb x^{k+1},\bb x^k)  \bb f^\tau(\bb x^{k+1},\bb x^k) &= \sum_{i,j=1}^n d_i\left( \dfrac{\xi_i}{x_i^k}g\left(\dfrac{x_i^{k+1}}{x_i^k}\right)-1\right)x_i^\tau a_{ij}x_j^\tau\left(1- \dfrac{\xi_j}{x_j^k}g\left(\dfrac{x_j^{k+1}}{x_j^k}\right)\right) \\
&=-\bb y^T DA \bb y = 0,
\] where the components of $\bb y$ are given by $y_i=x_i^\tau\left(1- \dfrac{\xi_i}{x_i^k} g\left(\dfrac{x_i^{k+1}}{x_i^k}\right)\right)$. 

Thus, a conservative scheme for~\eqref{LVsys} is given by~\eqref{LVDisc}. Now by Lemma \ref{lem:sym}, since $\dfrac{1}{x_j^k}g\left(\dfrac{x_j^{k+1}}{x_j^k}\right)$ is symmetric with $h(\bb x^k) = x_j^k$ for all $j=1,\dots,n$, the discretization \eqref{LVDisc} is symmetric provided $x_i^\tau, x_i^\tau$ are symmetric. Specifically, the symmetric choices $x_i^\tau := \frac{1}{2}(x_i^k+x_i^{k+1})$ or $\sqrt{x_i^{k}x_i^{k+1}}$ were selected in \eqref{LVDisc}.

\subsection{Many--body problem with pairwise radial potentials}
\label{app:manybody}

Next, we verify the conserved quantities of \eqref{nBodySys} and show the derived scheme is conservative and symmetric.

One can verify \eqref{nBodySys} has the conserved quantities~\eqref{eq:ConservedQuantitiesManyBodyProblem} using \eqref{discMultCond1}. Indeed, define the vector of conserved quantities $\bb \psi$ as
\vskip -4mm \[
\bb \psi(t,\bb q, \bb p) := \begin{pmatrix} H(\bb q, \bb p) \\ \bb P(\bb q, \bb p) \\ \bb L(\bb q, \bb p) \\ \bb C(t,\bb q, \bb p)\end{pmatrix}.
\]\vskip -1mm Then by \eqref{multCond1}, the associated $10 \times (6n)$ multiplier matrix is given by,
\[
\Lambda(t,\bb q, \bb p) :=
\begin{pmatrix}
\begin{bmatrix}\sum\limits_{j=1, j\neq i}^n \dfrac{\partial V_{ij}}{\partial q_{ij}}\dfrac{1}{q_{ij}}(\bb q_i-\bb q_j)\end{bmatrix}_{1\leq i\leq n}^T & \begin{bmatrix}\dfrac{1}{m_i}\bb p_i\end{bmatrix}_{1\leq i\leq n}^T \\
\begin{bmatrix}0_{3\times 3}\end{bmatrix}_{1\leq i\leq n}^T & \begin{bmatrix}I_{3\times 3}\end{bmatrix}_{1\leq i\leq n}^T \\
-[\Omega(\bb p_i)]_{1\leq i\leq n}^T & [\Omega(\bb q_i)]_{1\leq i\leq n}^T \\
\begin{bmatrix}\dfrac{m_i}{M}I_{3\times 3}\end{bmatrix}_{1\leq i\leq n}^T & -\begin{bmatrix}\dfrac{t}{M}I_{3\times 3}\end{bmatrix}_{1\leq i\leq n}^T
\end{pmatrix},
\] where $\Omega(\bb x)$ denotes the skew-symmetric $3\times 3$ matrix,
\[
\Omega(\bb x) := \begin{bmatrix}
0 & -z & y \\
z & 0 & -x \\
-y & x & 0 
\end{bmatrix},
\]associated with the cross product of $\bb x=(x,y,z)^T\in\mathbb{R}^3$ such that $\Omega(\bb x) \bb y = \bb x \times \bb y$ for $\bb y \in\mathbb{R}^3$.
Thus, \eqref{multCond2} is satisfied since,

\[
\Lambda \bb f &= 
\begin{pmatrix}
\sum\limits_{1\leq i\neq j \leq n} \dfrac{\partial V_{ij}}{\partial q_{ij}}\dfrac{1}{q_{ij}}\left[\dfrac{1}{m_i}\bb p_i^T(\bb q_i-\bb q_j)-(\bb q_i-\bb q_j)^T\dfrac{\bb p_i}{m_i}\right] \\
\sum\limits_{1\leq i\neq j \leq n} \dfrac{\partial V_{ij}}{\partial q_{ij}}\dfrac{1}{q_{ij}}(\bb q_i-\bb q_j) \\
-\sum\limits_{1\leq i \leq n} \dfrac{1}{m_i}\Omega(\bb p_i) \bb p_i + \sum\limits_{1\leq i\neq j \leq n}\dfrac{\partial V_{ij}}{\partial q_{ij}}\dfrac{1}{q_{ij}}\Omega(\bb q_i)(\bb q_i-\bb q_j) \\
\dfrac{1}{M}\sum\limits_{1\leq i \leq n} \bb p_i - t \sum\limits_{1\leq i \leq n}\dfrac{\partial V_{ij}}{\partial q_{ij}}\dfrac{1}{q_{ij}}(\bb q_i-\bb q_j)
\end{pmatrix}\\
&=\begin{pmatrix}
0 \\
\sum\limits_{1\leq i<j \leq n} \dfrac{\partial V_{ij}}{\partial q_{ij}}\left(\dfrac{\bb q_i-\bb q_j}{q_{ij}}+\dfrac{\bb q_j-\bb q_i}{q_{ji}}\right) \\
-\sum\limits_{1\leq i \leq n} \dfrac{1}{m_i}\bb p_i\times\bb p_i + \sum\limits_{1\leq i<j \leq n}\dfrac{\partial V_{ij}}{\partial q_{ij}}\left(\dfrac{\bb q_i}{q_{ij}}\times(\bb q_i-\bb q_j)+\dfrac{\bb q_j}{q_{ji}}\times(\bb q_j-\bb q_i)\right) \\
\dfrac{\bb P}{M}- t \sum\limits_{1\leq i \leq n}\dfrac{\partial V_{ij}}{\partial q_{ij}}\dfrac{1}{q_{ij}}\left(\dfrac{\bb q_i-\bb q_j}{q_{ij}}+\dfrac{\bb q_j-\bb q_i}{q_{ji}}\right)
\end{pmatrix}=\begin{pmatrix}
0 \\
\bb 0 \\
\bb 0 \\
\dfrac{\bb P}{M}
\end{pmatrix}\\
&= -\partial_t \bb \psi,
\] where the second last equality follows from $q_{ji}=q_{ij}, \bb q_j - \bb q_i = -(\bb q_i - \bb q_j)$ and properties of the cross-product. \\

We now employ DMM to derive conservative schemes for \eqref{nBodySys}. For simplicity, we will propose consistent choices of $D_t^\tau \bb x$, $D_t^\tau \bb \psi$, $\partial_t^\tau \bb \psi$, $\Lambda^\tau$ and $\bb f^\tau$ such that conditions \eqref{discMultCond1} and \eqref{discMultCond2} are satisfied. To accomplish this, we define $D_t^\tau \bb x$, $D_t^\tau \bb \psi$ and $\partial_t^\tau \bb \psi$ by,
\[
D_t^\tau \bb x := \dfrac{1}{\tau}\begin{pmatrix} \Delta \bb q \\ \Delta \bb p \end{pmatrix}, \hskip 5mm
D_t^\tau \bb \psi := \dfrac{1}{\tau}\begin{pmatrix}  \Delta H  \\
\Delta \bb P \\
\Delta \bb L \\
\Delta \bb C
 \end{pmatrix}, \hskip 5mm
\partial_t^\tau \bb \psi := \begin{pmatrix} 0\\ \bb 0 \\ \bb 0 \\ -\dfrac{\overline{\bb P}}{M} \end{pmatrix}.
\]
For the discrete multiplier $\Lambda^\tau$ and the discrete right hand side $\bb f^\tau$, we define
\[
\Lambda^\tau(t^{k+1}, \bb q^{k+1}, \bb p^{k+1}, t^{k}, \bb q^{k}, \bb p^{k}) &:=
\begin{pmatrix}
\begin{bmatrix}\sum\limits_{j=1, j\neq i}^n \dfrac{\Delta V_{ij}}{\Delta q_{ij}}\dfrac{1}{\overline{q_{ij}}}(\overline{\bb q_i}-\overline{\bb q_j})\end{bmatrix}_{1\leq i\leq n}^T & \begin{bmatrix}\dfrac{1}{m_i}\overline{\bb p_i}\end{bmatrix}_{1\leq i\leq n}^T \\
\begin{bmatrix}0_{3\times 3}\end{bmatrix}_{1\leq i\leq n}^T & \begin{bmatrix}I_{3\times 3}\end{bmatrix}_{1\leq i\leq n}^T \\
-[\Omega(\overline{\bb p_i})]_{1\leq i\leq n}^T & [\Omega(\overline{\bb q_i})]_{1\leq i\leq n}^T \\
\begin{bmatrix}\dfrac{m_i}{M}I_{3\times 3}\end{bmatrix}_{1\leq i\leq n}^T & -\begin{bmatrix}\dfrac{t^{k+1}+t^k}{2M}I_{3\times 3}\end{bmatrix}_{1\leq i\leq n}^T
\end{pmatrix}, \\
\bb f^\tau(t^{k+1}, \bb q^{k+1}, \bb p^{k+1}, t^{k}, \bb q^{k}, \bb p^{k}) &:= \begin{pmatrix}
 \begin{bmatrix}\dfrac{\overline{\bb p_i}}{m_i}\end{bmatrix}_{1\leq i\leq n}  \\
 -\begin{bmatrix}\sum\limits_{j=1, j\neq i}^n \dfrac{\Delta V_{ij}}{\Delta q_{ij}}\dfrac{1}{\overline{q_{ij}}}(\overline{\bb q_i}-\overline{\bb q_j})\end{bmatrix}_{1\leq i\leq n}
 \end{pmatrix}.
\] 
It can be seen that $\Lambda^\tau$ and $\bb f^\tau$ are consistent to $\Lambda$ and $\bb f$. We now verify the above choices satisfy condition \eqref{discMultCond2}.

\begin{align}
\Lambda^\tau \bb f^\tau &= 
\begin{pmatrix}
\sum\limits_{1\leq i\neq j \leq n} \dfrac{\Delta V_{ij}}{\Delta q_{ij}}\dfrac{1}{\overline{q_{ij}}}\left[\dfrac{1}{m_i}\overline{\bb p_i}^T(\overline{\bb q_i}-\overline{\bb q_j})-(\overline{\bb q_i}-\overline{\bb q_j})^T\dfrac{\overline{\bb p_i}}{m_i} \right] \\
\sum\limits_{1\leq i\neq j \leq n} \dfrac{\Delta V_{ij}}{\Delta q_{ij}}\dfrac{1}{\overline{q_{ij}}}(\overline{\bb q_i}-\overline{\bb q_j}) \\
-\sum\limits_{1\leq i \leq n} \dfrac{1}{m_i}\Omega(\overline{\bb p_i}) \overline{\bb p_i} + \sum\limits_{1\leq i\neq j \leq n}\dfrac{\Delta V_{ij}}{\Delta q_{ij}}\dfrac{1}{\overline{q_{ij}}}\Omega(\overline{\bb q_i})(\overline{\bb q_i}-\overline{\bb q_j}) \\
\dfrac{1}{M}\sum\limits_{1\leq i \leq n} \overline{\bb p_i} - \dfrac{t^{k+1}+t^k}{2M} \sum\limits_{1\leq i\neq j \leq n}\dfrac{\Delta V_{ij}}{\Delta q_{ij}}\dfrac{1}{\overline{q_{ij}}}(\overline{\bb q_i}-\overline{\bb q_j})
\end{pmatrix} \nonumber\\
&=\begin{pmatrix}
0 \\
\sum\limits_{1\leq i<j \leq n} \dfrac{\Delta V_{ij}}{\Delta q_{ij}}\left(\dfrac{\overline{\bb q_i}-\overline{\bb q_j}}{\overline{q_{ij}}}+\dfrac{\overline{\bb q_j}-\overline{\bb q_i}}{\overline{q_{ji}}}\right) \\
-\sum\limits_{1\leq i \leq n} \dfrac{1}{m_i}\overline{\bb p_i}\times\overline{\bb p_i} + \sum\limits_{1\leq i<j \leq n}\dfrac{\Delta V_{ij}}{\Delta q_{ij}}\left(\dfrac{\overline{\bb q_i}}{\overline{q_{ij}}}\times(\overline{\bb q_i}-\overline{\bb q_j})+\dfrac{\overline{\bb q_j}}{\overline{q_{ji}}}\times(\overline{\bb q_j}-\overline{\bb q_i})\right) \\
\dfrac{\overline{\bb P}}{M}- \dfrac{t^{k+1}+t^k}{2M} \sum\limits_{1\leq i<j \leq n}\dfrac{\Delta V_{ij}}{\Delta q_{ij}}\dfrac{1}{\overline{q_{ij}}}\left(\dfrac{\overline{\bb q_i}-\overline{\bb q_j}}{\overline{q_{ij}}}+\dfrac{\overline{\bb q_j}-\overline{\bb q_i}}{\overline{q_{ji}}}\right)
\end{pmatrix} \label{eq:manybody_appendix_calc}\\
&=\begin{pmatrix}
0 \\
\bb 0 \\
\bb 0 \\
\dfrac{\overline{\bb P}}{M}
\end{pmatrix} = -\partial_t^\tau \bb \psi.\nonumber
\end{align}
By direct computation or using divided difference calculus, one can show that
\[
\Delta\left( \dfrac{\bb p_i^T \bb p_i}{2} \right) &= \overline{\bb p_i}^T\Delta \bb p_i ,\\
\Delta V_{ij }&= \dfrac{\Delta V_{ij}}{\Delta q_{ij}}\dfrac{1}{\overline{q_{ij}}}(\overline{\bb q_i}-\overline{\bb q_j})^T\Delta (\overline{\bb q_i}-\overline{\bb q_j}), \\
\Delta \left( \bb q_i \times \bb p_i\right) &= \overline{\bb q_i} \times \Delta \bb p_i+\Delta \bb q_i\times \overline{\bb p_i}, \\
\Delta(t\bb p_i) &= \left(\dfrac{t^{k+1}+t^k}{2}\right)\Delta \bb p_i+ \tau \overline{\bb p_i}.
\]
So by linearity of the forward difference $\Delta$, condition \eqref{discMultCond1} also holds since,
\[
\Lambda^\tau D_t^\tau \bb x &= 
\dfrac{1}{\tau}\begin{pmatrix}
\sum\limits_{1\leq i \leq n}\dfrac{1}{m_i}\overline{\bb p_i}^T\Delta \bb p_i+\sum\limits_{1\leq i\neq j \leq n} \dfrac{\Delta V_{ij}}{\Delta q_{ij}}\dfrac{1}{\overline{q_{ij}}}(\overline{\bb q_i}-\overline{\bb q_j})^T\Delta \bb q_i \\
\sum\limits_{1\leq i \leq n} \Delta \bb p_i \\
\sum\limits_{1\leq i \leq n} \overline{\bb q_i} \times \Delta \bb p_i-\overline{\bb p_i} \times \Delta \bb q_i \\
\dfrac{1}{M}\sum\limits_{1\leq i \leq n} m_i \Delta \bb q_i - \dfrac{1}{M} \sum\limits_{1\leq i \leq n} \left(\dfrac{t^{k+1}+t^k}{2}\right)\Delta \bb p_i
\end{pmatrix} \\
&= \dfrac{1}{\tau}\begin{pmatrix}
\sum\limits_{1\leq i \leq n}\Delta\left(\dfrac{\bb p_i^T\bb p_i}{2 m_i}\right)+\sum\limits_{1\leq i< j \leq n} \dfrac{\Delta V_{ij}}{\Delta q_{ij}}\dfrac{1}{\overline{q_{ij}}}(\overline{\bb q_i}-\overline{\bb q_j})^T\Delta (\overline{\bb q_i}-\overline{\bb q_j}) \\
 \Delta \left( \sum\limits_{1\leq i \leq n} \bb p_i\right) \\
\sum\limits_{1\leq i \leq n} \Delta \left( \bb q_i \times \bb p_i\right) \\
\Delta\left(\dfrac{1}{M}\sum\limits_{1\leq i \leq n} m_i \Delta \bb q_i\right) - \dfrac{1}{M}\sum\limits_{1\leq i \leq n} \Delta(t \bb p_i) -\tau\overline{\bb p_i}
\end{pmatrix}\\
&=\dfrac{1}{\tau}\begin{pmatrix}
\sum\limits_{1\leq i \leq n}\Delta\left(\dfrac{\bb p_i^T\bb p_i}{2 m_i}\right)+\sum\limits_{1\leq i< j \leq n} \Delta V_{ij} \\
 \Delta \bb P \\
 \Delta \bb L \\
\Delta\left(\dfrac{1}{M}\left(\sum\limits_{1\leq i \leq n} m_i \bb q_i\right) - \dfrac{P}{M}t \right) +\dfrac{\tau}{M} \overline{\bb P}\end{pmatrix}
=\dfrac{1}{\tau}\begin{pmatrix}
\Delta H \\
\Delta \bb P \\
\Delta \bb L \\
\Delta \bb C
\end{pmatrix} + \begin{pmatrix}
0 \\
\bb 0 \\
\bb 0 \\
\dfrac{\overline{\bb P}}{M}
\end{pmatrix} =
D_t^\tau \bb \psi -\partial_t^\tau \bb \psi.
\]
Thus, the discretization \eqref{nBodyDisc} is indeed a conservative scheme for~\eqref{nBodySys}. Moreover, since $\overline{\bb p_i}, \overline{\bb q_i}$ and $\overline{q_{ij}}$ are symmetric under the permutation of $k\leftrightarrow k+1$, \eqref{nBodyDisc} is a symmetric scheme provided that $\dfrac{\Delta V_{ij}}{\Delta q_{ij}}$ is symmetric under the permutation of $k\leftrightarrow k+1$. Indeed, this is true by the definition of the divided difference of $\dfrac{\Delta V_{ij}}{\Delta q_{ij}}$ and that $V_{ij}$ is a function of $q_{ij}$.

\subsection{Point vortex problem: Planar case}

Here, we verify the conserved quantities of \eqref{pvPlane} and show the derived scheme is conservative and symmetric using Lemma \ref{lem:sym}.

First, we verify that~\eqref{eq:ConservedQuantitiesPlanarPointVortex} are indeed conserved quantities of \eqref{pvPlane} by defining the conserved vector $\bb \psi$ as,
\[
\bb \psi(\bb x, \bb y) := \begin{pmatrix} 
\bb P(\bb x, \bb y) \\
L(\bb x, \bb y) \\
H(\bb x, \bb y)
\end{pmatrix}.
\] 
Then by \eqref{multCond1}, the associated $4 \times (2n)$ multiplier matrix is given by,
\[
\Lambda(\bb x, \bb y) := \begin{pmatrix}
\begin{bmatrix}\Gamma_i\end{bmatrix}_{1\leq i\leq n}^T & \begin{bmatrix} 0\end{bmatrix}_{1\leq i\leq n}^T \\
\begin{bmatrix}0\end{bmatrix}_{1\leq i\leq n}^T & \begin{bmatrix}\Gamma_i\end{bmatrix}_{1\leq i\leq n}^T \\
\begin{bmatrix}2\Gamma_i x_i\end{bmatrix}_{1\leq i\leq n}^T & \begin{bmatrix}2\Gamma_i y_i\end{bmatrix}_{1\leq i\leq n}^T \\
\begin{bmatrix}-\dfrac{1}{2\pi}\Gamma_i\sum\limits_{j=1, j\neq i}^n \Gamma_j \dfrac{x_{ij}}{ r_{ij}^2}\end{bmatrix}_{1\leq i\leq n}^T & \begin{bmatrix}-\dfrac{1}{2\pi}\Gamma_i\sum\limits_{j=1, j\neq i}^n \Gamma_j \dfrac{y_{ij}}{ r_{ij}^2}\end{bmatrix}_{1\leq i\leq n}^T
\end{pmatrix}
\]
So \eqref{multCond2} is satisfied since,
\[
\Lambda(\bb x, \bb y) \bb f(\bb x, \bb y) &= \begin{pmatrix}
\dfrac{1}{2\pi}\sum\limits_{1\leq i\neq j\leq n} \Gamma_i\Gamma_j \dfrac{y_{ij}}{ r_{ij}^2} \\
-\dfrac{1}{2\pi}\sum\limits_{1\leq i\neq j\leq n} \Gamma_i\Gamma_j \dfrac{x_{ij}}{ r_{ij}^2} \\
\dfrac{1}{\pi}\sum\limits_{1\leq i \neq j\leq n} \Gamma_i \Gamma_j \dfrac{y_{ij}x_i-x_{ij}y_i}{ r_{ij}^2} \\
-\dfrac{1}{2\pi}\sum\limits_{1\leq i\leq n} \Gamma_i \sum\limits_{j,l=1, j,l\neq i}^n \Gamma_j\Gamma_l \left(\dfrac{y_{ij}}{ r_{ij}^2}\dfrac{x_{il}}{ r_{il}^2} - \dfrac{x_{ij}}{ r_{ij}^2} \dfrac{y_{il}}{ r_{il}^2} \right)
\end{pmatrix} \\
&=
\begin{pmatrix}
\dfrac{1}{2\pi}\sum\limits_{1\leq i<j\leq n} \Gamma_i\Gamma_j \left(\dfrac{y_{ij}}{ r_{ij}^2}+\dfrac{y_{ji}}{ r_{ji}^2}\right) \\
-\dfrac{1}{2\pi}\sum\limits_{1\leq i<j\leq n} \Gamma_i\Gamma_j \left(\dfrac{x_{ij}}{ r_{ij}^2}+\dfrac{x_{ji}}{ r_{ji}^2}\right) \\
\dfrac{1}{\pi}\sum\limits_{1\leq i<j\leq n} \Gamma_i \Gamma_j \left(\dfrac{y_{ij}x_i-x_{ij}y_i}{ r_{ij}^2}+\dfrac{y_{ji}x_j-x_{ji}y_j}{ r_{ji}^2}\right) \\
-\dfrac{1}{2\pi}\sum\limits_{1\leq i\leq n} \Gamma_i \left(\sum\limits_{j,l=1, j,l\neq i}^n \Gamma_j\Gamma_l  \dfrac{y_{ij}}{ r_{ij}^2}\dfrac{x_{il}}{ r_{il}^2}- \sum\limits_{j,l=1, j,l\neq i}^n \Gamma_j\Gamma_l \dfrac{x_{ij}}{ r_{ij}^2} \dfrac{y_{il}}{ r_{il}^2} \right)
\end{pmatrix} = \bb 0,
\] where the last equality follows from 
\[
r_{ji} &= r_{ij}, \\
y_{ji}&=-y_{ij}, \\
x_{ji}&=-x_{ij}, \\
y_{ji}x_j-x_{ji}y_j &= y_ix_j-x_i y_j=-(y_{ij}x_i-x_{ij}y_i).
\]

Similar to the $n$-body problem, we will propose consistent choices of of $D_t^\tau \bb x$, $D_t^\tau \bb \psi$, $\partial_t^\tau \bb \psi$, $\Lambda^\tau$ and $\bb f^\tau$ and verify that both conditions \eqref{discMultCond1} and \eqref{discMultCond2} are satisfied. Analogous to the $n$-body problem, we define
\[
D_t^\tau \bb x := \dfrac{1}{\tau}\begin{pmatrix} \Delta \bb x \\ \Delta \bb y \end{pmatrix}, \hskip 5mm D_t^\tau \bb \psi := \dfrac{1}{\tau}\begin{pmatrix}
\Delta \bb P \\
\Delta L \\
\Delta H
 \end{pmatrix}, \hskip 5mm
\partial_t^\tau \bb \psi := \bb 0.
\]
Let us define the discrete multiplier $\Lambda^\tau$ and the discrete right hand side $\bb f^\tau$ as,

\[
&\Lambda^\tau(\bb x^{k+1}, \bb y^{k+1}, \bb x^k, \bb y^k) := \\
&\hskip 10mm\begin{pmatrix}
\begin{bmatrix}\Gamma_i\end{bmatrix}_{1\leq i\leq n}^T & \begin{bmatrix} 0\end{bmatrix}_{1\leq i\leq n}^T \\
\begin{bmatrix}0\end{bmatrix}_{1\leq i\leq n}^T & \begin{bmatrix}\Gamma_i\end{bmatrix}_{1\leq i\leq n}^T \\
\begin{bmatrix}2\Gamma_i \overline{x_i}\end{bmatrix}_{1\leq i\leq n}^T & \begin{bmatrix}2\Gamma_i \overline{y_i}\end{bmatrix}_{1\leq i\leq n}^T \\
\begin{bmatrix}-\dfrac{1}{2\pi}\Gamma_i\sum\limits_{j=1, j\neq i}^n \Gamma_j \dfrac{\overline{x_{ij}}}{(r_{ij}^k)^2}g\left(z_{ij}\right) \end{bmatrix}_{1\leq i\leq n}^T & \begin{bmatrix}-\dfrac{1}{2\pi}\Gamma_i\sum\limits_{j=1, j\neq i}^n \Gamma_j \dfrac{\overline{y_{ij}}}{(r_{ij}^k)^2}g\left(z_{ij}\right)\end{bmatrix}_{1\leq i\leq n}^T
\end{pmatrix},\\
&\bb f^\tau(\bb x^{k+1}, \bb y^{k+1}, \bb x^k, \bb y^k) := \begin{pmatrix} 
-\begin{bmatrix}
\dfrac{1}{2\pi}\sum\limits_{1\leq j \leq n, j\neq i} 
\dfrac{\Gamma_j \overline{y_{ij}}}{( r_{ij}^k)^2}g\left(z_{ij}\right) \end{bmatrix}_{1\leq i\leq n} \\ \begin{bmatrix}
\dfrac{1}{2\pi}\sum\limits_{1\leq j \leq n, j\neq i}  \dfrac{\Gamma_j \overline{x_{ij}}}{( r_{ij}^k)^2}g\left(z_{ij}\right) \end{bmatrix}_{1\leq i\leq n}
\end{pmatrix},
\] where for brevity we have denoted $\displaystyle z_{ij} := \left(\dfrac{ r_{ij}^{k+1}}{ r_{ij}^{k}}\right)^2$. Since $\displaystyle g(z_{ij}) \rightarrow 1$ as $z_{ij} \rightarrow 1$ when $\tau \rightarrow 0$, $\Lambda^\tau$ and $\bb f^\tau$ are consistent to $\Lambda$ and $\bb f$. Next, we verify condition \eqref{discMultCond2}.
\[\Lambda^\tau \bb f^\tau
&= \begin{pmatrix}
\dfrac{1}{2\pi}\sum\limits_{1\leq i\neq j\leq n} \Gamma_i\Gamma_j \dfrac{\overline{y_{ij}}}{(r_{ij}^k)^2}g\left(z_{ij}\right) \\
-\dfrac{1}{2\pi}\sum\limits_{1\leq i\neq j\leq n} \Gamma_i\Gamma_j \dfrac{\overline{x_{ij}}}{(r_{ij}^k)^2}g\left(z_{ij}\right) \\
\dfrac{1}{\pi}\sum\limits_{1\leq i \neq j\leq n} \Gamma_i \Gamma_j \dfrac{\overline{y_{ij}}\hskip 1mm\overline{x_i}-\overline{x_{ij}}\hskip 1mm\overline{y_i}}{(r_{ij}^k)^2}g\left(z_{ij}\right) \\
-\dfrac{1}{2\pi}\sum\limits_{1\leq i\leq n} \Gamma_i \sum\limits_{j,l=1, j,l\neq i}^n \Gamma_j\Gamma_l \left(\dfrac{\overline{y_{ij}}}{(r_{ij}^k)^2}\dfrac{\overline{x_{il}}}{(r_{il}^k)^2}- \dfrac{\overline{x_{ij}}}{(r_{ij}^k)^2} \dfrac{\overline{y_{il}}}{(r_{il}^k)^2} \right)g\left(z_{ij}\right)g\left(z_{il}\right)\end{pmatrix}
\]
\[
&=
\begin{pmatrix}
\dfrac{1}{2\pi}\sum\limits_{1\leq i<j\leq n} \Gamma_i\Gamma_j \left(\dfrac{\overline{y_{ij}}}{(r_{ij}^k)^2}g\left(z_{ij}\right)+\dfrac{\overline{y_{ji}}}{(r_{ji}^k)^2}g\left(z_{ji}\right)\right) \\
-\dfrac{1}{2\pi}\sum\limits_{1\leq i<j\leq n} \Gamma_i\Gamma_j \left(\dfrac{\overline{x_{ij}}}{(r_{ij}^k)^2}g\left(z_{ij}\right)+\dfrac{\overline{x_{ji}}}{(r_{ji}^k)^2}g\left(z_{ji}\right)\right) \\
\dfrac{1}{\pi}\sum\limits_{1\leq i<j\leq n} \Gamma_i \Gamma_j \left(\dfrac{\overline{y_{ij}}\hskip 1mm\overline{x_i}-\overline{x_{ij}}\hskip 1mm \overline{y_i}}{(r_{ij}^k)^2}g\left(z_{ij}\right)+\dfrac{\overline{y_{ji}}\hskip 1mm \overline{x_j}-\overline{x_{ji}}\hskip 1mm \overline{y_j}}{(r_{ji}^k)^2}g\left(z_{ji}\right)\right) \\
-\dfrac{1}{2\pi}\sum\limits_{1\leq i\leq n} \Gamma_i \left(\sum\limits_{j,l=1, j,l\neq i}^n \Gamma_j\Gamma_l \dfrac{\overline{y_{ij}}}{(r_{ij}^k)^2}\dfrac{\overline{x_{il}}}{(r_{il}^k)^2}g\left(z_{ij}\right)g\left(z_{il}\right) - \sum\limits_{j,l=1, j,l\neq i}^n \Gamma_j\Gamma_l  \dfrac{\overline{x_{ij}}}{(r_{ij}^k)^2} \dfrac{\overline{y_{il}}}{(r_{il}^k)^2}g\left(z_{ij}\right)g\left(z_{il}\right) \right)
\end{pmatrix}\\
&= \bb 0,
\] where the last equality follows from
\[
r_{ji}^k &= r_{ij}^k, \\
z_{ji} &= z_{ij},\\
\overline{y_{ji}}&=-\overline{y_{ij}}, \\
\overline{x_{ji}}&=-\overline{x_{ij}}, \\
\overline{y_{ji}}\hskip 1mm\overline{x_j}-\overline{x_{ji}}\hskip 1mm\overline{y_j} &= \overline{y_i}\hskip 1mm\overline{x_j}-\overline{x_i}\hskip 1mm\overline{y_j}=-(\overline{y_{ij}}\hskip 1mm\overline{x_i}-\overline{x_{ij}}\hskip 1mm\overline{y_i}).
\]
To verify condition \eqref{discMultCond1}, it follows from direct computation or divided difference calculus that

\[
\Delta (x_i^2+y_i^2) &= 2\overline{x_i}\Delta x_i+2\overline{y_i}\Delta y_i,\\
\Delta (r_{ij})^2 &= 2\overline{x_i}\Delta x_i+2\overline{y_i}\Delta y_i,\\
\dfrac{g\left(z_{ij}\right)}{(r_{ij}^{k})^2} &= \dfrac{2(\log r_{ij}^{k+1}- 2\log r_{ij}^{k})}{\Delta (r_{ij})^2},\\
\Delta \log(r_{ij}) &= (\log r_{ij}^{k+1}- \log r_{ij}^{k})\dfrac{2\overline{x_{ij}}\Delta x_i+2\overline{y_{ij}}\Delta y_i}{\Delta (r_{ij})^2} = \left(\overline{x_{ij}}\Delta x_i+\overline{y_{ij}}\Delta y_i\right)\dfrac{g\left(z_{ij}\right)}{(r_{ij}^{k})^2}.
\]
Combining with linearity of $\Delta$, condition \eqref{discMultCond1} is satisfied since,

\[
\Lambda^\tau D_t^\tau \bb x &= \frac{1}{\tau}\begin{pmatrix}
\sum\limits_{1\leq i\leq n} \Gamma_i\Delta x_i \\
\sum\limits_{1\leq i\leq n} \Gamma_i\Delta y_i \\
\sum\limits_{1\leq i\leq n} \Gamma_i(2\overline{x_i}\Delta x_i+2\overline{y_i}\Delta y_i) \\
-\dfrac{1}{2\pi}\sum\limits_{1\leq i\neq j\leq n} \Gamma_i \Gamma_j \left(\overline{x_{ij}}\Delta x_i+\overline{y_{ij}}\Delta y_i\right)\dfrac{g\left(z_{ij}\right)}{(r_{ij}^k)^2}
\end{pmatrix}\]
\[
&= \frac{1}{\tau}\begin{pmatrix}
\Delta \left(\sum\limits_{1\leq i\leq n} \Gamma_ix_i\right) \\
\Delta \left(\sum\limits_{1\leq i\leq n} \Gamma_i y_i\right) \\
\Delta \left(\sum\limits_{1\leq i\leq n} \Gamma_i(x_i^2+y_i^2) \right) \\
-\dfrac{1}{2\pi}\sum\limits_{1\leq i<j\leq n} \Gamma_i \Gamma_j \left(\overline{x_{ij}}\Delta x_{ij}+\overline{y_{ij}}\Delta y_{ij}\right)\dfrac{g\left(z_{ij}\right)}{(r_{ij}^k)^2}
\end{pmatrix}
= \frac{1}{\tau}\begin{pmatrix} \Delta \bb P\\ \Delta L \\ \Delta H\end{pmatrix} = D_t^\tau \bb \psi - \partial_t^\tau \bb \psi.
\] 

Thus, the scheme \eqref{pvPlaneDisc} is indeed conservative for~\eqref{pvPlane}. Moreover, \eqref{pvPlaneDisc} is symmetric since $\overline{x}_{ij}, \overline{y}_{ij}$ and $\dfrac{1}{(r_{ij}^k)^2}g(z_{ij})= \dfrac{1}{(r_{ij}^k)^2}g\left(\left(\dfrac{r_{ij}^{k+1}}{r_{ij}^k}\right)^2\right)$ is symmetric under the permutation of $k\leftrightarrow k+1$, by Lemma \ref{lem:sym} with $h(\bb x^k) = (r_{ij}^k)^2$.

\subsection{Point vortex problem: Spherical case}

Finally, we verify the conserved quantities of \eqref{pvSphere} and show that the derived scheme \eqref{pvSphereDisc} is conservative and symmetric using Lemma \ref{lem:sym}.

The conserved vector $\bb \psi$ of \eqref{eq:ConservedQuantitiesSphericalPointVortex} for point vortex problem on the unit sphere \eqref{pvSphere} is
\[
\bb \psi(\bb x) := \begin{pmatrix} 
\bb P(\bb x) \\
H(\bb x)
\end{pmatrix}.
\] 
Using~\eqref{multCond1}, the associated $4 \times (3n)$ multiplier matrix is given by,
\[
\Lambda(\bb x) := \begin{pmatrix}
\begin{bmatrix}\Gamma_iI\end{bmatrix}_{1\leq i\leq n}^T\\
\begin{bmatrix}\dfrac{1}{4\pi}\Gamma_i\sum\limits_{j=1, j\neq i}^n \Gamma_j \dfrac{{\bb x}_i-{\bb x}_j}{1-{\bb x}_i\cdot{\bb x}_j}\end{bmatrix}_{1\leq i\leq n}^T
\end{pmatrix}
\]
where $I$ is the $3\times 3$ identity matrix. So \eqref{multCond2} is satisfied since,
\[
\Lambda(\bb x) \bb f(\bb x) &= \begin{pmatrix}
\dfrac{1}{4\pi}\sum\limits_{1\leq i\neq j\leq n} \Gamma_i\Gamma_j \dfrac{{\bb x}_j\times {\bb x}_i}{1-{\bb x}_j\cdot {\bb x}_i}\\
\dfrac{1}{4\pi}\Gamma_i\sum\limits_{j=1, j\neq i}^n \Gamma_j \dfrac{{\bb x}_i-{\bb x}_j}{1-{\bb x}_i\cdot{\bb x}_j}\cdot\dfrac{1}{4\pi}\sum\limits_{k=1, k\neq i}^n \Gamma_k \dfrac{{\bb x}_k\times {\bb x}_i}{1-{\bb x}_i\cdot {\bb x_k}}
\end{pmatrix} \\
&=
\begin{pmatrix}
\dfrac{1}{4\pi}\sum\limits_{1\leq i<j\leq n} \Gamma_i\Gamma_j \left(\dfrac{{\bb x}_j\times {\bb x}_i}{ 1-{\bb x}_j\cdot {\bb x}_i}+\dfrac{{\bb x}_i\times {\bb x}_j}{ 1-{\bb x}_j\cdot {\bb x}_i}\right) \\
\dfrac{1}{16\pi^2}\sum\limits_{i=1}^n\Gamma_i\sum\limits_{1\le j<k\le n}\Gamma_j\Gamma_k\left(\dfrac{{\bb x}_j\cdot ({\bb x}_k\times {\bb x}_i)+{\bb x}_k\cdot ({\bb x}_j\times {\bb x}_i)}{(1-{\bb x}_i\cdot {\bb x}_j)(1-{\bb x}_i\cdot {\bb x}_k)}\right)
\end{pmatrix} = \bb 0,
\]
where the last equality follows from the antisymmetry of the cross product and the antisymmetry of the scalar triple product.

As before, we choose $D_t^\tau \bb x$, $D_t^\tau \bb \psi$, $\partial_t^\tau \bb \psi$, $\Lambda^\tau$ and $\bb f^\tau$ such that both conditions \eqref{discMultCond1} and \eqref{discMultCond2} are satisfied. We start by defining
\[
D_t^\tau \bb x := \dfrac{\Delta \bb x}{\tau}, \hskip 5mm D_t^\tau \bb \psi := \dfrac{1}{\tau}\begin{pmatrix}
\Delta \bb P \\
\Delta H
 \end{pmatrix}, \hskip 5mm
\partial_t^\tau \bb \psi := \bb 0.
\]
The discrete multiplier $\Lambda^\tau$ and the discrete right hand side $\bb f^\tau$ are defined as,
\[
&\Lambda^\tau(\bb x^{k+1},\bb x^k) := \begin{pmatrix}
\begin{bmatrix}\Gamma_iI\end{bmatrix}_{1\leq i\leq n}^T\\
\begin{bmatrix}\dfrac{1}{4\pi}\Gamma_i\sum\limits_{j=1, j\neq i}^n \Gamma_j \dfrac{\overline{{\bb x}_{ij}}}{1-{\bb x}_i^k\cdot{\bb x}_j^k}g(z_{ij})\end{bmatrix}_{1\leq i\leq n}^T
\end{pmatrix},\\
&\bb f^\tau(\bb x^{k+1},\bb x^k) := 
\dfrac{1}{4\pi}\sum\limits_{1\leq j \leq n, j\neq i} 
\Gamma_j \dfrac{\overline{{\bb x}_j}\times\overline{{\bb x}_i}}{1-{{\bb x}_i^k}\cdot{{\bb x}_j^k}}g\left(z_{ij}\right),
\] 
where for brevity we denoted $\displaystyle z_{ij} :=\dfrac{1-{\bb x}_i^{k+1}\cdot{\bb x}_j^{k+1}}{ 1-{\bb x}_i^{k}\cdot{\bb x}_j^{k}}$. As in the planar case, $\displaystyle g(z_{ij}) \rightarrow 1$ as $z_{ij}\rightarrow 1$ when $\tau \rightarrow 0$ and so both $\Lambda^\tau$ and $\bb f^\tau$ are consistent to $\Lambda$ and $\bb f$, respectively.

Similar to the planar case, it is readily verified that we have $\Lambda^\tau \bb f^\tau=\bb 0$ and thus it remains to check condition~\eqref{discMultCond1}. Indeed, this condition is satisfied as
\[
 \Lambda^\tau D_t^\tau \bb x &= \frac{1}{\tau}\begin{pmatrix}
\sum\limits_{1\leq i\leq n} \Gamma_i\Delta {\bb x}_i \\
\dfrac{1}{4\pi}\sum\limits_{1\leq i\neq j\leq n} \Gamma_i \Gamma_j\dfrac{\overline{{\bb x}_{ij}}\cdot\Delta {\bb x}_{i}}{1-{\bb x}_i^k\cdot{\bb x}_j^k}g\left(z_{ij}\right)
\end{pmatrix} \\
&= \frac{1}{\tau}\begin{pmatrix}
\Delta \left(\sum\limits_{1\leq i\leq n} \Gamma_i{\bb x}_i\right) \\
\dfrac{1}{4\pi}\sum\limits_{1\leq i<j\leq n} \Gamma_i \Gamma_j \dfrac{\overline{{\bb x}_{ij}}\cdot\Delta {\bb x}_{ij}}{1-{\bb x}_i^k\cdot{\bb x}_j^k}g\left(z_{ij}\right)
\end{pmatrix}
= \frac{1}{\tau}\begin{pmatrix} \Delta \bb P\\ \Delta H\end{pmatrix} = D_t^\tau \bb \psi - \partial_t^\tau \bb \psi.
\] 
Hence, the discretization \eqref{pvSphereDisc} is conservative for~\eqref{pvSphere}. Finally, \eqref{pvSphereDisc} is symmetric since $\overline{\bb x}_{j}\times \overline{\bb x}_{i}$, and $\dfrac{1}{1-{\bb x}_i^k\cdot{\bb x}_j^k}g(z_{ij}) = \dfrac{1}{1-{\bb x}_i^k\cdot{\bb x}_j^k}g\left(\dfrac{1-{\bb x}_i^{k+1}\cdot{\bb x}_j^{k+1}}{ 1-{\bb x}_i^{k}\cdot{\bb x}_j^{k}}\right)$ is symmetric under the permutation of $k\leftrightarrow k+1$, by Lemma \ref{lem:sym} with $h(\bb x^k) = 1-{\bb x}_i^k\cdot{\bb x}_j^k$.

\end{document}